\documentclass[11pt,twoside, leqno]{article}

\usepackage{mathrsfs}
\usepackage{amssymb}
\usepackage{amsmath}
\usepackage{amsthm}
\usepackage{amsfonts}
\usepackage{color}

\allowdisplaybreaks

\def\rr{{\mathbb R}}
\def\rn{{{\rr}^n}}

\def\fz{\infty}

\def\az{\alpha}

\renewcommand\tilde{\widetilde}

\def\supp{{\rm{\,supp\,}}}

\def\ls{\lesssim}
\def\lz{\lambda}

\def\bz{\beta}

\def\gz{{\gamma}}

\def\sz{\sigma}

\def\hs{\hspace{0.3cm}}

\def\r{\right}
\def\lf{\left}

\def\bint{{\ifinner\rlap{\bf\kern.30em--}
\int\else\rlap{\bf\kern.35em--}\int\fi}\ignorespaces}

\def\sbint{{\ifinner\rlap{\bf\kern.32em--}
\hspace{0.078cm}\int\else\rlap{\bf\kern.45em--}\int\fi}\ignorespaces}

\newtheorem{theorem}{Theorem}[section]
\newtheorem{lemma}[theorem]{Lemma}
\newtheorem{corollary}[theorem]{Corollary}

\theoremstyle{definition}

\newtheorem{definition}[theorem]{Definition}
\numberwithin{equation}{section}

\numberwithin{equation}{section}

\numberwithin{equation}{section}

\begin{document}

\arraycolsep=1pt

\title{\Large\bf Parameterized Littlewood-Paley operators with variable kernels on Hardy spaces and weak Hardy spaces
\footnotetext{\hspace{-0.35cm}
{\it 2010 Mathematics Subject Classification}.
{Primary 42B25; Secondary 42B30.}
\endgraf{\it Key words and phrases.}  Littlewood-Paley operator, Hardy space, variable kernel.
%\endgraf This work is partially supported by the National
%Natural Science Foundation of China (Grant Nos. 11461065 \& 11661075) and A Cultivate Project for Young Doctor from Xinjiang Uyghur Autonomous Region (No. qn2015bs003).
% \endgraf $^\ast$\,Corresponding author
}}
\author{Li Bo
%\\
%{{\small{\it School of Mathematics, Tianjin University, Tianjin, 300354, P. R. China}}
%}
}
\date{  }
\maketitle

\vspace{-0.8cm}

\begin{minipage}{14.5cm}\small
{
\noindent
{\bf Abstract.}
In this paper, by using the atomic decomposition theory of Hardy space and weak Hardy space,
we discuss the boundedness of parameterized Littlewood-Paley operator with variable kernel
on these spaces.
}
%{\bf{Key words:}} Marcinkiewicz integral; Muckenhoupt weight; Musielak-Orlicz function; Hardy space \\
%{ \bf{MR(2010) Subject Classification:}} {42B20; 42B30; 46E30} / {\bf{CLC number:}} O174.2
\end{minipage}

%%%%%%%%%%%%%%%%%%%%%%%%%%%%%%%%%%%%%%%%%%%%%%%

%%%%%%%%%%%%%%% Section 1 %%%%%%%%%%%%%%%%%%%%%

%%%%%%%%%%%%%%%%%%%%%%%%%%%%%%%%%%%%%%%%%%%%%%%
\section{Introduction}\label{s1}
Let $S^{n-1}$ be the unit sphere in the $n$-dimensional Euclidean space $\rn \ (n\ge2)$ with normalized Lebesgue measure $d\sigma$.
A function $\Omega(x,\,z)$ defined on $\rn\times\rn$ is said to be in $L^\fz(\rn)\times L^q(S^{n-1})$ with $q\ge1$,
if $\Omega(x,\,z)$ satisfies the following conditions:
\begin{align}\label{e1.1}
\Omega(x,\,\lz z)=\Omega(x,\,z) \ \mathrm{for} \ \mathrm{any} \ x,\,z\in\rn \ \mathrm{and} \ \lz>0,
\end{align}
\begin{align}\label{e1.2}
\int_{S^{n-1}}\Omega(x,\,z)\,d\sigma(z')=0 \ \mathrm{for} \ \mathrm{any} \ x\in\rn,
\end{align}
\begin{align}\label{e1.3}
\sup_{\substack{x\in\rn \\ r\ge0}}\lf(\int_{S^{n-1}}|\Omega(x+rz',\,z')|^q\,d\sigma(z')\r)^{1/q}<\fz,
\end{align}
where $z':=z/{|z|}$ for any $z\neq{\mathbf{0}}$.
%Set $K(x,\,z):=\frac{\Omega(x,\,z')}{|z|^n}$ for all $(x,\,z)\in\rn\times\rn$.
The {\it {singular integral operator with variable kernel}} is defined by
\begin{align*}
T_\Omega(f)(x):=\mathrm{p.v.}\int_\rn \frac{\Omega(x,\,x-y)}{|x-y|^n}f(y)\,dy.
\end{align*}

In 1955 and 1956, Calder\'{o}n and Zygmund \cite{cz55,cz56} investigated the $L^p$ boundedness of $T_\Omega$.
They found that these operators are closely related to the problem about the second-order linear elliptic
equations with variable coefficients.
%In \cite{cz78}, they generalized \cite[Theorem 2]{cz55} and \cite[Theorem 2]{cz56} as follows.
%For any index $1\le q\le\fz$, $q'$ denotes the {\it{conjugate index}} of $q$, namely, $1/q+1/{q'}=1$.
%\begin{flushleft}
%{\bf{Theorem A.}}
%{\it{
%Let $p,\,q\in(1,\,\fz)$ satisfy
%\begin{enumerate}
%\item[\rm{(i)}] $\frac{1}{q}<\frac{1}{p'}+\frac{1}{p'(n-1)}$ if $1<p\le2$; or
%\item[\rm{(ii)}] $\frac{1}{q}<\frac{1}{p'}+\frac{1}{p(n-1)}$ if $2\le p<\fz$,
%\end{enumerate}
%where $p':=p/(p-1)$. Suppose that $\Omega(x,\,z)\in L^\fz(\rn) \times L^q(S^{n-1})$.
%Then there exists a positive constant $C$ independent of $f$ such that
%$$\|T_\Omega(f)\|_{L^p(\rn)}\le C\|f\|_{L^p(\rn)}.$$
%In particular, $T_\Omega$ is bounded on $L^p(\rn)$ for $p\ge q'$.
%}}
%\end{flushleft}
In 2011, Chen and Ding \cite{cd11} consider the same problem for
the {\it{parameterized Littlewood-Paley operators with variable kernel}}
$\mu_{\Omega,\,S}^\rho$ and $\mu^{\rho,\,\ast}_{\Omega,\,\lz}$ defined by, respectively,
%setting, for any $x\in\rn$ and $\rho\in(0,\,n)$,
\begin{align*}
\mu^\rho_{\Omega,\,S}(f)(x):=\lf(\int\int_{|y-x|<t}\lf|
\int_{|y-z|<t}\frac{\Omega(y,\,y-z)}{|y-z|^{n-\rho}}f(z)\,dz\r|^2\,\frac{dydt}{t^{n+2\rho+1}}\r)^{1/2}
\end{align*}
and
\begin{align*}
\mu^{\rho,\,\ast}_{\Omega,\,\lz}(f)(x):
=\lf[\int\int_{{\mathbb{R}}^{n+1}_+}\lf(\frac{t}{t+|x-y|}\r)^{\lz n}\lf|
\int_{|y-z|<t}\frac{\Omega(y,\,y-z)}{|y-z|^{n-\rho}}f(z)\,dz\r|^2\,\frac{dydt}{t^{n+2\rho+1}}\r]^{1/2},
\end{align*}
where $0<\rho<n$ and $1<\lz<\fz$.
They obtained the following result:
\begin{flushleft}
{\bf{Theorem A.}} (\cite[Theorem 5.1]{cd11})
{\it{
Let $0<\rho<n$, $1<\lz<\fz$ and $4\le p<\fz$.
If $\Omega(x,z)\in L^\fz(\rn)\times L^2(S^{n-1})$,
then there exists a positive constant $C$ independent of $f$ such that
\begin{align*}
\lf\|\mu^\rho_{\Omega,\,S}(f)\r\|_{L^p(\rn)}\le 2^{\lz n}\lf\|\mu^{\rho,\,\ast}_{\Omega,\,\lz}(f)\r\|_{L^p(\rn)}\le C\lf\|f\r\|_{L^p(\rn)}.
\end{align*}
}}
\end{flushleft}

On the other hand, as everyone knows, many important operators
are better behaved on Hardy space $H^p(\rn)$ than on Lebesgue $L^p(\rn)$
space in the range $p\in(0,\,1]$.
For example, when $p\in(0,\,1]$, the Riesz transforms are bounded on Hardy space $H^p(\rn)$,
but not on the corresponding Lebesgue space $L^p(\rn)$.
Therefore, one can consider $H^p(\rn)$ to be a very natural replacement for $L^p(\rn)$ when $p\in(0,\,1]$.
We refer to \cite{l95} for a complete survey of the real-variable theory of Hardy space.
Motivated by this, the question arises, when $p\in(0,\,1]$,
whether the operators $\mu_{\Omega,\,S}^\rho$ and $\mu^{\rho,\,\ast}_{\Omega,\,\lz}$ are bounded
from Hardy space $H^p(\rn)$ to Lebesgue space $L^p(\rn)$.
In this paper we shall answer this problem affirmatively. Not only that,
we also discuss boundedness of $\mu_{\Omega,\,S}^\rho$ and $\mu^{\rho,\,\ast}_{\Omega,\,\lz}$ from weak Hardy space $WH^p(\rn)$
to weak Lebesgue space $WL^p(\rn)$.
%More conclusions of operator with variable kernel are referred to \cite{cch15,cd08,ty02}.

Precisely, the present paper is built up as follows.
In next section, we first recall some notions concerning variable kernel $\Omega(x,\,z)$.
Then we discuss the boundedness of $\mu_{\Omega,\,S}^\rho$ from (weak) Hardy space to (weak) Lebesgue space
(see Theorems \ref{dl.1}, \ref{dl.2} and \ref{dl.3} below).
Section \ref{s3} is devoted to establishing the boundedness of $\mu^{\rho,\,\ast}_{\Omega,\,\lz}$
from (weak) Hardy space to (weak) Lebesgue space (see Theorems \ref{dl.4}, \ref{dl.5} and \ref{dl.6} below).
In the last section, we will give some remarks for the above conclusions.
%
%Finally, we make some conventions on notation.
%Let $\zz_+:=\{1,\, 2,\,\ldots\}$ and $\nn:=\{0\}\cup\zz_+$.
%For any $\bz:=(\bz_1,\ldots,\bz_n)\in\nn^n$,
%let $|\bz|:=\bz_1+\cdots+\bz_n$.
Throughout this paper the letter $C$ will denote a \emph{positive constant} that may vary
from line to line but will remain independent of the main variables.
The \emph{symbol} $P\ls Q$ stands for the inequality $P\le CQ$. If $P\ls Q\ls P$, we then write $P\sim Q$.
For any $1\le q\le\fz$, $q'$ denotes the {\it{conjugate index}} of $q$, namely, $1/q+1/{q'}=1$.
%For any sets $E,\,F \subset \rn$, we use $E^\complement$ to denote the set $\rn\setminus E$,
%$|E|$ its  {\it{$n$-dimensional Lebesgue measure}},
%$\chi_E$ its \emph{characteristic function} and
%$E+F$ the {\it algebraic sum} $\{x+y:\ x\in E,\,y\in F\}$.
%For any $s\in\rr$, $\lfloor s\rfloor$ denotes the
%unique integer such that $s-1<\lfloor s\rfloor\le s$.
%\emph{maximal integer} not larger than $s$.
%If there are no special instructions, any space $\mathcal{X}(\rn)$ is denoted simply by $\mathcal{X}$. For instance, $L^2(\rn)$ is simply denoted by $L^2$.
%For any index $q\in[1,\,\fz]$, $q'$ denotes the {\it{conjugate index}} of $q$, namely, $1/q+1/{q'}=1$.
%For any set $E$ of $\rn$, $t\in[0,\,\infty)$ and measurable function $\vz$,
%let $\vz(E,\,t):=\int_E\vz(x,\,t)\,dx$ and $\{|f|>t\}:=\{x\in\rn: \ |f(x)|>t\}$.
%As usual we use $B_r$ to denote the ball $\{x\in\rn:\ |x|<r\}$ with $r\in(0,\,\fz)$.

%%%%%%%%%%%%%%%%%%%%%%%%%%%%%%%%%%%%%%%%%%%%%%%

%%%%%%%%%%%%%%% Section 2 %%%%%%%%%%%%%%%%%%%%%

%%%%%%%%%%%%%%%%%%%%%%%%%%%%%%%%%%%%%%%%%%%%%%%

\section{Boundedness of $\mu_{\Omega,\,S}^\rho$ on $H^p(\rn)$ and $WH^p(\rn)$}\label{s2}

Before stating the main results of this scetion, we recall some notions about
the variable kernel $\Omega(x,\,z)$.
For any $0<\alpha\le1$, a function $\Omega(x,z)\in L^\fz(\rn)\times L^2(S^{n-1})$ is said to satisfy the
{\it{$L^{2,\,\az}$-Dini condition}}
%(when $\alpha=0$, it is called the {\it{$L^1$-Dini condition}})
if
\begin{align*}
\int_0^1\frac{\omega_2(\delta)}{\delta^{1+\alpha}}\,d\delta<\fz,
\end{align*}
where
%where ${\omega(\delta)}$ is the integral modulus of continuity of $\Omega$ defined by
\begin{align*}
{\omega_2(\delta)}:=\sup_{\substack{x\in\rn \\ r\ge0}}
\lf(\int_{S^{n-1}}\sup_{\substack{y'\in S^{n-1} \\ |y'-z'|\le\delta}}
\lf|\Omega(x+rz',\,y')-\Omega(x+rz',\,z')\r|^2\,d\sigma(z')\r)^{1/2}.
\end{align*}
On the other hand, for any $0<\az\le1$, a function $\Omega(x,z)\in L^\fz(\rn)\times L^2(S^{n-1})$ is said to satisfy the
{\it{Lipschitz condition of order $\az$}} if there exists a positive constant $C$ such that,
for any $x\in\rn$ and $y',\,z'\in S^{n-1}$,
\begin{align*}
|\Omega(x,\,y')-\Omega(x,\,z')|\le C|y'-z'|^\az.
\end{align*}
It is noteworthy that the relationship between $L^{2,\,\az}$-Dini condition and Lipschitz condition of order $\az$ is not clear up to now.
%For any $\alpha,\,\beta\in(0,\,1]$ with $\beta<\alpha$,
%it is trivial to see that if $\Omega$ satisfies the $L^{q,\,\az}$-Dini condition,
%then it also satisfies the $L^{q,\,\bz}$-Dini condition.
%We thus denote by ${\rm{Din}}^q_\alpha$ the class of all functions which satisfy the
%$L^{q,\,\bz}$-Dini conditions for all $\beta<\alpha$.
%For any $\alpha\in(0,\,1]$, we define
%$${\rm{Din}}^\fz_\alpha:=\bigcap_{q\ge1}{\rm{Din}}^q_\alpha.$$
%A routine computation gives rise to
%$${\rm{Din}}^r_\alpha \subset {\rm{Din}}^q_\alpha \ \ \mathrm{if} \ 1\le q< r\le \fz,$$
%and
%$${\rm{Din}}^q_\alpha \subset {\rm{Din}}^q_\bz \ \ \mathrm{if} \ 0<\bz < \az\le 1.$$

The main results of this section are as follows.
\begin{theorem}\label{dl.1}
Let $0<\az\le1$, $n/2<\rho<n$, $0<\bz<\min\{1/2,\,\az,\,\rho-n/2\}$ and $n/(n+\bz)<p<1$.
Suppose $\Omega(x,\,z)$ satisfies the $L^{2,\,\az}$-Dini condition or the Lipschitz condition of order $\az$.
Then $\mu_{\Omega,\,S}^\rho$ is bounded from $H^p(\rn)$ to $L^p(\rn)$.
\end{theorem}

%\begin{theorem}\label{dl.2}
%Let $0<\az\le1$, $n/2<\rho<n$, $0<\bz<\min\{1/2,\,\az,\,\rho-n/2\}$ and $n/(n+\bz)<p<1$.
%Suppose $\Omega$ satisfies the Lipschitz condition of order $\az$.
%Then $\mu_{\Omega,\,S}^\rho$ is bounded from $H^p(\rn)$ to $L^p(\rn)$.
%\end{theorem}

\begin{theorem}\label{dl.2}
Let $n/2<\rho<n$ and $\Omega(x,\,z)\in L^\fz(\rn) \times L^2(S^{n-1})$.
If
\begin{align*}
\int_0^1\frac{\omega_2(\delta)}{\delta}(1+|\log{\delta}|)^\sz\,d\delta<\fz \ \mathrm{for \ some} \ \sz>1,
\end{align*}
then $\mu_\Omega^\rho$ is bounded from $H^1(\rn)$ to $L^1(\rn)$.
\end{theorem}

\begin{theorem}\label{dl.3}
Let $0<\az\le1$, $n/2<\rho<n$, $0<\bz<\min\{1/2,\,\az,\,\rho-n/2\}$ and $n/(n+\bz)<p\le1$.
Suppose $\Omega(x,\,z)$ satisfies the $L^{2,\,\az}$-Dini condition or the Lipschitz condition of order $\az$.
Then $\mu_{\Omega,\,S}^\rho$ is bounded from $WH^p(\rn)$ to $WL^p(\rn)$.
\end{theorem}

%\begin{theorem}\label{dl.5}
%Let $0<\az\le1$, $n/2<\rho<n$, $0<\bz<\min\{1/2,\,\az,\,\rho-n/2\}$ and $n/(n+\bz)<p\le1$.
%Suppose $\Omega$ satisfies the Lipschitz condition of order $\az$.
%Then $\mu_{\Omega,\,S}^\rho$ is bounded from $WH^p(\rn)$ to $WL^p(\rn)$.
%\end{theorem}

%\begin{theorem}\label{dl.3}
%Let $0<\rho<n$ and $q>2(n-1)/n$. Suppose that $\Omega(x,\,z)\in L^\fz(\rn) \times L^q(S^{n-1})$ satisfies the $L^1$-Dini condition.
%Then $\mu_\Omega^\rho$ is bounded from $H^1(\rn)$ to $L^1(\rn)$.
%\end{theorem}

%Before proving the main results of this section, we need some definitions and lemma
To show the above theorems, we need the following definition of atom.
\begin{definition}\label{d2.11}{\rm{(\cite{l95})}}
Let $0<p\le1$ and the nonnegative integer $s\ge \lfloor n(1/p-1)\rfloor$
($\lfloor x\rfloor$ denotes the integral part of real number $x$).
A function $a(x)$ is called a {\it $(p,\,\fz,\,s)$-atom} associated with some ball $B\subset\rn$
if it satisfies the following three conditions:
\begin{enumerate}
\item[\rm{(i)}] $a$ is supported in $B$;
\item[\rm{(ii)}] $\|a\|_{L^\fz(\rn)}\leq |B|^{-1/p}$;
\item[\rm{(iii)}] $\int_\rn a(x)x^\gz dx=0$ for any multi-index $\gz$ with $|\gz|\leq s$.
\end{enumerate}
\end{definition}

\begin{lemma}\label{l3.6}{\rm{(\cite{dll07})}}
Let $0<\rho<n$. Suppose $\Omega(x,\,z)\in L^\fz(\rn) \times L^2(S^{n-1})$.
If there exists a constant $0<\bz<1/2$ such that $|z|<\bz R$, then, for any $h\in\rn$,
\begin{align*}
\int_{R\le |y|<2R}\lf|\frac{\Omega(y+h,\,y-z)}{|y-z|^{n-\rho}}-\frac{\Omega(y+h,\,y)}{|y|^{n-\rho}}\r|\,dy
\leq CR^{\rho-n/2}\lf({\frac{|z|}{R}}+\int_{{2|z|}/{R}}^{{4|z|}/{R}}\frac{\omega_2(\delta)}{\delta}d\delta\r),
\end{align*}
where the positive constant $C$ is independent of $R$ and $y$.
\end{lemma}

\begin{lemma}\label{yl.1}
Let $0<\az\le1$, $n/2<\rho<n$ and $0<\bz<\min\{1/2,\,\az,\,\rho-n/2\}$.
Suppose $\Omega(x,\,z)$ satisfies the $L^{2,\,\az}$-Dini condition or the Lipschitz condition of order $\az$.
If $a(x)$ is a $(p,\,\fz,\,s)$-atom associated with some ball $B:=B(x_0,\,r)$,
then there exists a positive constant $C$ independent of $a(x)$ such that, for any $x\in (64B)^\complement$,
\begin{align*}
\mu^\rho_{\Omega,\,S}(a)(x)\le C\|a\|_{L^\fz(\rn)}\frac{r^{n+\beta}}{|x-x_0|^{n+\beta}}.
\end{align*}
\end{lemma}
\begin{proof}
We only consider the case $\Omega(x,\,z)$ satisfies the $L^{2,\,\az}$-Dini condition.
In another case, the proof is easier and we leave the details to the interested reader.
The trick of the proof is to find a subtle segmentation.
To be precise, for any $x\in (64B)^\complement$, let us write
\begin{align*}
\mu^\rho_{\Omega,\,S}(a)(x)
&=\lf(\int\int_{|y-x|<t}\lf|\int_{|y-z|<t}\frac{\Omega(y,\,y-z)}{|y-z|^{n-\rho}}a(z)\,dz\r|^2\,\frac{dydt}{t^{n+2\rho+1}}\r)^{1/2} \\
&\le \lf(\int\int_{\substack{|y-x|<t \\ y\in 16B}}\lf|\int_{|y-z|<t}\frac{\Omega(y,\,y-z)}{|y-z|^{n-\rho}}a(z)\,dz\r|^2\,\frac{dydt}{t^{n+2\rho+1}}\r)^{1/2} \\
&\hs+\lf(\int\int_{\substack{|y-x|<t \\ y\in (16B)^\complement \\ t\le|y-x_0|+8r}}\cdot\cdot\cdot\r)^{1/2}
+\lf(\int\int_{\substack{|y-x|<t \\ y\in (16B)^\complement \\ t>|y-x_0|+8r}}\cdot\cdot\cdot\r)^{1/2}
=:{\rm{I_1+I_2+I_3}}.
\end{align*}

We estimate ${\rm{I_1}}$ first. By $x\in (64B)^\complement$, $y\in 16B$ and $z\in B$, we know that
$$t>|y-x|\ge|x-x_0|-|y-x_0|>|x-x_0|-|x-x_0|/4>|x-x_0|/2 \ {\rm{and}} \ |y-z|<32r.$$
From this, Minkowski's inequality for integrals, $\Omega(x,\,z)\in L^\fz(\rn)\times L^2(S^{n-1})$ and $0<\bz<\rho-n/2$,
it follows that, for any $x\in (64B)^\complement$,
\begin{align*}
{\rm{I_1}}
&\le \int_B|a(z)|\lf(\int\int_{\substack{t>|x-x_0|/2 \\ |y-z|<32r}}\frac{|\Omega(y,\,y-z)|^2}{|y-z|^{2n-2\rho}}\frac{dydt}{t^{n+2\rho+1}}\r)^{1/2}\,dz\\
&= \int_B|a(z)|\lf(\int_{|y-z|<32r}\frac{|\Omega(y,\,y-z)|^2}{|y-z|^{2n-2\rho}}\,dy\r)^{1/2}\lf(\int_{|x-x_0|/2}^\fz\frac{dt}{t^{n+2\rho+1}}\r)^{1/2}\,dz \\
&\sim \frac{1}{|x-x_0|^{n/2+\rho}}\int_B|a(z)|\lf(\int_{|y|<32r}\frac{|\Omega(z+y,\,y)|^2}{|y|^{2n-2\rho}}\,dy\r)^{1/2}\,dz \\
&\sim \frac{1}{|x-x_0|^{n/2+\rho}}\int_B|a(z)|\lf[\int_{0}^{32r}\lf(\int_{S^{n-1}}|\Omega(z+uy',\,y')|^2\,d\sigma(y')\r)\frac{u^{n-1}}{u^{2n-2\rho}}\,du\r]^{1/2}\,dz \\
&\ls \frac{1}{|x-x_0|^{n/2+\rho}}\int_B|a(z)|\lf(\int_{0}^{32r}u^{2\rho-n-1}\,du\r)^{1/2}\,dz \\
&\ls \|a\|_{L^\fz(\rn)}\frac{r^{n/2+\rho}}{|x-x_0|^{n/2+\rho}}\ls \|a\|_{L^\fz(\rn)}\frac{r^{n+\bz}}{|x-x_0|^{n+\bz}},
%\lf(\int\int_{\substack{|y-x|<t \\ y\in 16B}}\lf|\int_{|y-z|<t}\frac{\Omega(y-z)}{|y-z|^{n-\rho}}b(z)\,dz\r|^2\,\frac{dydt}{t^{n+2\rho+1}}\r)^{1/2} \\
%&\le \lf[\int\int_{\substack{y\in 16B \\ t>|x-x_0|/2}}\lf(\int_{|y-z|<32r}\frac{|\Omega(y-z)|}{|y-z|^{n-\rho}}|b(z)|\,dz\r)^2\,\frac{dydt}{t^{n+2\rho+1}}\r]^{1/2} \\
%&\le \|\Omega\|_{L^\fz(S^{n-1})}\|b\|_{L^\fz}\lf[\int\int_{\substack{y\in 16B \\ t>|x-x_0|/2}}\lf(\int_{|z|<32r}\frac{1}{|z|^{n-\rho}}\,dz\r)^2\,\frac{dydt}{t^{n+2\rho+1}}\r]^{1/2} \\
%&\ls \|b\|_{L^\fz}\lf(\int_{16B}1\,dy\,\int^\fz_{\frac{|x-x_0|}{2}}\frac{dt}{t^{n+2\rho+1}}\r)^{1/2}\lf(\int_{S^{n-1}}\int^{32r}_0\frac{1}{u^{n-\rho}}u^{n-1}\,dud\sigma(z')\r) \\
%&\sim \|b\|_{L^\fz}\frac{r^{\rho+n/2}}{|x-x_0|^{\rho+n/2}}\ls \|b\|_{L^\fz}\frac{r^{n+\bz}}{|x-x_0|^{n+\bz}},
\end{align*}
which is wished.

Now we are interested ${\rm{I_2}}$. By $x\in (64B)^\complement$, $y\in (16B)^\complement$, $z\in B$ and the mean value theorem, we know that
\begin{align}\label{q1}
r<|y-z|\sim|y-x_0|;
\end{align}
\begin{align}\label{q3}
|x-x_0|\le|x-y|+|y-x_0|\le t+|y-x_0|\le 2|y-x_0|+8r\le 3|y-x_0|;
\end{align}
\begin{align}\label{q2}
|y-x_0|-2r\le|y-x_0|-|z-x_0|\le|y-z|<t\le|y-x_0|+8r;
\end{align}
\begin{align}\label{q4}
\lf|\frac{1}{(|y-x_0|-2r)^{n+2\rho}}-\frac{1}{(|y-x_0|+8r)^{n+2\rho}}\r|\ls\frac{r}{|y-x_0|^{n+2\rho+1}}.
\end{align}
From Minkowski's inequality for integrals, \eqref{q1}-\eqref{q4}, $\Omega(x,\,z)\in L^\fz(\rn)\times L^2(S^{n-1})$ and $\bz<1/2$, we deduce that, for any $x\in (64B)^\complement$,
\begin{align*}
{\rm{I_2}}
&=\lf(\int\int_{\substack{|y-x|<t \\ y\in (16B)^\complement \\ t\le|y-x_0|+8r}}\lf|\int_{|y-z|<t}\frac{\Omega(y,\,y-z)}{|y-z|^{n-\rho}}a(z)\,dz\r|^2\,\frac{dydt}{t^{n+2\rho+1}}\r)^{1/2} \\
&\le \int_B|a(z)| \lf(\int\int_{\substack{|y-x|<t \\ y\in (16B)^\complement \\ t\le|y-x_0|+8r \\ |y-z|<t}}\frac{|\Omega(y,\,y-z)|^2}{|y-z|^{2n-2\rho}}\,
\frac{dydt}{t^{n+2\rho+1}}\r)^{1/2}\,dz \\
&\le \int_B|a(z)| \lf[\int_{\substack{|y-z|>r \\ |x-x_0|\le3|y-x_0|}}\frac{|\Omega(y,\,y-z)|^2}{|y-z|^{2n-2\rho}}
\lf(\int^{|y-x_0|+8r}_{|y-x_0|-2r}\frac{dt}{t^{n+2\rho+1}}\r)\,dy\r]^{1/2}\,dz \\
&\ls \int_B|a(z)| \lf(\int_{\substack{|y-z|>r \\ |x-x_0|\le3|y-x_0|}}\frac{|\Omega(y,\,y-z)|^2}{|y-z|^{2n-2\rho}}
\frac{r}{|y-x_0|^{n+2\rho+1}}\,dy\r)^{1/2}\,dz \\
&\ls \int_B|a(z)| \lf(\int_{\substack{|y-z|>r \\ |x-x_0|\le3|y-x_0|}}\frac{|\Omega(y,\,y-z)|^2}{|y-z|^{n-2\bz+1}}
\frac{r}{|x-x_0|^{2n+2\bz}}\,dy\r)^{1/2}\,dz \\
&\ls \frac{r^{1/2}}{|x-x_0|^{n+\bz}} \int_B |a(z)|\lf(\int_{|y|>r}\frac{|\Omega(z+y,\,y)|^2}{|y|^{n-2\bz+1}}\,dy\r)^{1/2}\,dz \\
&\ls \frac{r^{1/2}}{|x-x_0|^{n+\bz}} \int_B |a(z)|\lf(\int_r^\fz\frac{du}{u^{1-2\bz+1}}\r)^{1/2}\,dz \\
&\ls \|a\|_{L^\fz(\rn)}\frac{r^{n+\bz}}{|x-x_0|^{n+\bz}},
\end{align*}
which is also wished.

It remains to estimate ${\rm{I_3}}$.
It is apparent from $t>|y-x_0|+8r$ that $B\subset \{z\in\rn: \ |y-z|<t\}$.
%noticing that $t>|y-x_0|+8r$, we see that, for any $y\in (16B)^\complement$,
%\begin{align*}
%B\subset \{z\in\rn: \ |z-y|<t\}.
%\end{align*}
%On the other hand, we claim that, for any $y\in (16B)^\complement$ and $z\in B$,
%\begin{align}\label{q6}
%\lf|\frac{\Omega(y-z)}{|y-z|^{n-\rho}}-\frac{\Omega(y-x_0)}{|y-x_0|^{n-\rho}}\r|\ls \frac{|z-x_0|^\az}{|y-x_0|^{n-\rho+\az}}.
%\end{align}
%Indeed, by the mean value theorem and $\Omega(x,\,z)\in{\rm{Lip}}_\alpha(S^{n-1})$ with $\az\in(0,\,1]$, we obtain that, for any $y\in (16B)^\complement$ and $z\in B$,
%\begin{align*}
%\lf|\frac{\Omega(y-z)}{|y-z|^{n-\rho}}-\frac{\Omega(y-x_0)}{|y-x_0|^{n-\rho}}\r|
%&\le\lf|\frac{\Omega(y-z)}{|y-z|^{n-\rho}}-\frac{\Omega(y-z)}{|y-x_0|^{n-\rho}}\r|+\lf|\frac{\Omega(y-z)}{|y-x_0|^{n-\rho}}-\frac{\Omega(y-x_0)}{|y-x_0|^{n-\rho}}\r| \\
%&\ls\lf|\frac{1}{|y-z|^{n-\rho}}-\frac{1}{|y-x_0|^{n-\rho}}\r| \\
%&\hs+\frac{1}{|y-x_0|^{n-\rho}} \lf|\Omega\lf(\frac{y-z}{|y-z|}\r)-\Omega\lf(\frac{y-x_0}{|y-x_0|}\r)\r| \\
%&\ls\frac{|z-x_0|}{|y-x_0|^{n-\rho+1}}+\frac{1}{|y-x_0|^{n-\rho}} \lf|\frac{y-z}{|y-z|}-\frac{y-x_0}{|y-x_0|}\r|^\az \\
%&\ls\frac{1}{|y-x_0|^{n-\rho}}\frac{|z-x_0|}{|y-x_0|}+\frac{1}{|y-x_0|^{n-\rho}}\lf(\frac{|z-x_0|}{|y-x_0|}\r)^\az \\
%&\ls\frac{|z-x_0|^\az}{|y-x_0|^{n-\rho+\az}}.
%\end{align*}
By this, the vanishing moments of atom $a(z)$, and Minkowski's inequality for integrals, we obtain that, for any $x\in (64B)^\complement$,
\begin{align*}
{\rm{I_3}}
&=\lf[\int\int_{\substack{|y-x|<t \\ y\in (16B)^\complement \\ t>|y-x_0|+8r}}\lf|\int_{|y-z|<t}
\lf(\frac{\Omega(y,\,y-z)}{|y-z|^{n-\rho}}-\frac{\Omega(y,\,y-x_0)}{|y-x_0|^{n-\rho}}\r)a(z)\,dz\r|^2\frac{dydt}{t^{n+2\rho+1}}\r]^{1/2} \\
&\le\int_B |a(z)|
\lf(\int\int_{\substack{y\in (16B)^\complement \\ t>\max\{|y-x|,\,|y-x_0|+8r,\,|y-z|\}}}
\lf|\frac{\Omega(y,\,y-z)}{|y-z|^{n-\rho}}-\frac{\Omega(y,\,y-x_0)}{|y-x_0|^{n-\rho}}\r|^2\frac{dydt}{t^{n+2\rho+1}}\r)^{1/2}dz \\
&\le\int_B |a(z)|
\lf(\int\int_{\substack{y\in (16B)^\complement \\ t>\max\{|y-x|,\,|y-x_0|+8r,\,|y-z|\} \\ |x-x_0|\le3|y-x_0|}}
\lf|\frac{\Omega(y,\,y-z)}{|y-z|^{n-\rho}}-\frac{\Omega(y,\,y-x_0)}{|y-x_0|^{n-\rho}}\r|^2\frac{dydt}{t^{n+2\rho+1}}\r)^{1/2}dz \\
&\hs+\int_B |a(z)|
\lf(\int\int_{\substack{y\in (16B)^\complement \\ t>\max\{|y-x|,\,|y-x_0|+8r,\,|y-z|\} \\ |x-x_0|>3|y-x_0|}}
\cdot\cdot\cdot\r)^{1/2}dz=:{\rm{I_{31}}+\rm{I_{32}}}.
\end{align*}

Below, we will give the estimates of ${\rm{I_{31}}}$ and ${\rm{I_{32}}}$, respectively.

For ${\rm{I_{31}}}$, Lemma \ref{l3.6} and the assumption that $\Omega(x,\,z)$ satisfies the $L^{2,\,\az}$-Dini condition
yield that, for any $x\in (64B)^\complement$,
\begin{align*}
{\rm{I_{31}}}
&\le\int_B |a(z)|
\lf[\int_{\substack{y\in (16B)^\complement \\  |x-x_0|\le3|y-x_0|}}
\lf|\frac{\Omega(y,\,y-z)}{|y-z|^{n-\rho}}-\frac{\Omega(y,\,y-x_0)}{|y-x_0|^{n-\rho}}\r|^2\lf(\int^\fz_{|y-x_0|}\frac{dt}{t^{n+2\rho+1}}\r)dy\r]^{1/2}dz \\
&\ls\frac{\|a\|_{L^\fz(\rn)}}{|x-x_0|^{n+\bz}}\int_B
\lf(\int_{\substack{y\in (16B)^\complement \\  |x-x_0|\le3|y-x_0|}}
\lf|\frac{\Omega(y,\,y-z)}{|y-z|^{n-\rho}}-\frac{\Omega(y,\,y-x_0)}{|y-x_0|^{n-\rho}}\r|^2\frac{1}{|y-x_0|^{2\rho-n-2\bz}}dy\r)^{1/2}dz \\
&\ls\frac{\|a\|_{L^\fz(\rn)}}{|x-x_0|^{n+\bz}}\int_B
\lf(\int_{y\in (16B)^\complement }
\lf|\frac{\Omega(y,\,y-z)}{|y-z|^{n-\rho}}-\frac{\Omega(y,\,y-x_0)}{|y-x_0|^{n-\rho}}\r|^2\frac{1}{|y-x_0|^{2\rho-n-2\bz}}dy\r)^{1/2}dz \\
&\ls{\|a\|_{L^\fz(\rn)}}\frac{1}{|x-x_0|^{n+\bz}} \\
&\hs\times\int_B\sum_{j=4}^\fz\lf(\int_{2^j\le|y-x_0|<2^{j+1}r}
\lf|\frac{\Omega(y,\,y-z)}{|y-z|^{n-\rho}}-\frac{\Omega(y,\,y-x_0)}{|y-x_0|^{n-\rho}}\r|^2\frac{1}{|y-x_0|^{2\rho-n-2\bz}}dy\r)^{1/2}dz \\
&\sim{\|a\|_{L^\fz(\rn)}}\frac{1}{|x-x_0|^{n+\bz}} \\
&\hs\times\int_B\sum_{j=4}^\fz\frac{1}{(2^jr)^{\rho-n/2-\bz}}\lf(\int_{2^j\le|y-x_0|<2^{j+1}r}
\lf|\frac{\Omega(y,\,y-z)}{|y-z|^{n-\rho}}-\frac{\Omega(y,\,y-x_0)}{|y-x_0|^{n-\rho}}\r|^2dy\r)^{1/2}dz \\
&\ls{\|a\|_{L^\fz(\rn)}}\frac{1}{|x-x_0|^{n+\bz}}
\int_B\sum_{j=4}^\fz\frac{(2^jr)^{\rho-n/2}}{(2^jr)^{\rho-n/2-\bz}}\lf(\frac{|z-x_0|}{2^jr}
 +\int_{\frac{2|z-x_0|}{2^jr}}^\frac{4|z-x_0|}{2^jr}\frac{\omega_2(\delta)}{\delta}\,d\delta\r)dz \\
&\ls{\|a\|_{L^\fz(\rn)}}\frac{1}{|x-x_0|^{n+\bz}}
\int_B\sum_{j=4}^\fz\frac{(2^jr)^{\rho-n/2}}{(2^jr)^{\rho-n/2-\bz}}\lf[\frac{|z-x_0|}{2^jr}
 +\lf(\frac{|z-x_0|}{2^jr}\r)^\az\r]dz \\
&\ls{\|a\|_{L^\fz(\rn)}}\frac{1}{|x-x_0|^{n+\bz}}
\int_B\sum_{j=4}^\fz{(2^jr)^{\bz}}2^{-j\az}dz
\sim\|a\|_{L^\fz(\rn)}\frac{r^{n+\bz}}{|x-x_0|^{n+\bz}},
\end{align*}
where the last ``$\sim$" is due to $\bz<\az$.

For ${\rm{I_{32}}}$, noticing that $t>\max\{|y-x|,\,|y-x_0|+8r,\,|y-z|\}$ and $|x-x_0|>3|y-x_0|$, we see that
$$t>|y-x|\ge|x-x_0|-|y-x_0|>|x-x_0|/2.$$
From this, $\bz<\rho-n/2$ and the argument same as in ${\rm{I_{31}}}$, it follows that, for any $x\in (64B)^\complement$,
\begin{align*}
{\rm{I_{32}}}
&\le\int_B |a(z)|
\lf[\int_{y\in (16B)^\complement}
\lf|\frac{\Omega(y,\,y-z)}{|y-z|^{n-\rho}}-\frac{\Omega(y,\,y-x_0)}{|y-x_0|^{n-\rho}}\r|^2\lf(\int_{\substack{t>|y-x_0| \\ t>|x-x_0|/2}}\frac{dt}{t^{n+2\rho+1}}\r)dy\r]^{1/2}dz \\
&\ls\frac{\|a\|_{L^\fz(\rn)}}{|x-x_0|^{n+\bz}}\int_B
\lf[\int_{y\in (16B)^\complement}
\lf|\frac{\Omega(y,\,y-z)}{|y-z|^{n-\rho}}-\frac{\Omega(y,\,y-x_0)}{|y-x_0|^{n-\rho}}\r|^2\lf(\int_{|y-x_0|}^\fz\frac{dt}{t^{2\rho-n-2\bz+1}}\r)dy\r]^{1/2}dz \\
&\sim\frac{\|a\|_{L^\fz(\rn)}}{|x-x_0|^{n+\bz}}\int_B
\lf(\int_{y\in (16B)^\complement }
\lf|\frac{\Omega(y,\,y-z)}{|y-z|^{n-\rho}}-\frac{\Omega(y,\,y-x_0)}{|y-x_0|^{n-\rho}}\r|^2\frac{1}{|y-x_0|^{2\rho-n-2\bz}}dy\r)^{1/2}dz \\
&\ls\|a\|_{L^\fz(\rn)}\frac{r^{n+\bz}}{|x-x_0|^{n+\bz}}.
\end{align*}

Collecting the estimates of ${\rm{I_1}}$, ${\rm{I_2}}$, ${\rm{I_{31}}}$ and ${\rm{I_{32}}}$, we obtain the desired inequality.
This finishes the proof of Lemma \ref{yl.1}.
\end{proof}

\begin{proof}[Proof of Theorem \ref{dl.1}]
By the atomic decomposition theory of Hardy space (see \cite[Chapter 2]{l95}),
our problem reduces to prove that there exists a positive constant $C$ such that,
for any $(p,\,\fz,\,s)$-atom $a(x)$ associated with some ball $B:=B(x_0,\,r)$, $\|\mu_{\Omega,\,S}^\rho(a)\|_{L^p(\rn)}\le C$.
To this end, we estimate $\mu_{\Omega,\,S}^\rho(a)$ separately around and away from the support of atom $a(x)$.
More precisely, H\"{o}lder's inequality, the $L^4$ boundedness of $\mu_{\Omega,\,S}^\rho$ (see Theorem A), Lemma \ref{yl.1} and $p>n/(n+\bz)$ yield that
\begin{align*}
\int_{\rn}\lf|\mu_{\Omega,\,S}^\rho(a)(x)\r|^p\,dx
&=\int_{64B}\lf|\mu_{\Omega,\,S}^\rho(a)(x)\r|^p\,dx+\int_{(64B)^\complement}\lf|\mu_{\Omega,\,S}^\rho(a)(x)\r|^p\,dx \\
&\le\lf(\int_{64B}\lf|\mu_{\Omega,\,S}^\rho(a)(x)\r|^4\,dx\r)^{p/4}|64B|^{1-p/4}+\int_{(64B)^\complement}\lf|\mu_{\Omega,\,S}^\rho(a)(x)\r|^p\,dx \\
&\ls\|a\|_{L^\fz(\rn)}^p|B|+\|a\|^p_{L^\fz(\rn)}\int_{(64B)^\complement}\frac{r^{(n+\bz)p}}{|x-x_0|^{(n+\bz)p}}\,dx\ls1.
\end{align*}

The proof is completed.
\end{proof}

\begin{proof}[Proof of Theorem \ref{dl.2}]
Proceeding as in the proof of \cite[Theorem 1]{dlx07a},
it is quite believable that \cite[Theorem 1]{dlx07a} may also be true for the variable kernel case,
but to limit the length of this paper, we leave the details to the interested reader.
\end{proof}

To show Theorem \ref{dl.3}, we need the following atomic decomposition theory of weak Hardy space.

\begin{lemma}\label{whp}{\rm{(\cite{l95})}}
Let $0<p\le1$. For every $f\in WH^p(\rn)$,
there exists a sequence of bounded measurable functions $\{f_k\}_{k=-\fz}^\fz$ such that
\begin{enumerate}
\item [\rm{(i)}] $f=\sum_{k=-\fz}^\fz f_k$ in the sense of distributions.

\item [\rm{(ii)}] Each $f_k$ can be further decomposed into $f_k=\sum_i b^k_i$ and $\{b^k_i\}$ satisfies

 {\rm{\quad(a)}} $\supp{(b^k_i)}\subset B^k_i:=B(x^k_i,\,r^k_i)$;
Moreover, $\sum_i \chi_{B^k_i}(x)\le C$ and $\sum_i |B^k_i|\le c\,2^{-kp}$, where $c\sim\|f\|^p_{WH^p(\rn)}$;

 {\rm{\quad(b)}} $\|b^k_i\|_{L^\fz(\rn)}\le C2^k$, where $C$ is independent of $k$ and $i$;

 {\rm{\quad(c)}} $\int_\rn b^k_i(x)x^\gz\,dx=0$ for any multi-index $\gz$ with $|\gz|\leq \lfloor n(1/p-1)\rfloor$.
\end{enumerate}

Conversely, if distribution $f$ has a decomposition satisfying $\mathrm{(i)}$ and $\mathrm{(ii)}$, then $f\in WH^p(\rn)$.
Moreover, we have $\|f\|^p_{WH^p(\rn)}\sim c$.
\end{lemma}

\begin{proof}[Proof of Theorem \ref{dl.3}]
To show Theorem \ref{dl.3}, it suffices to prove that
there exist a positive constant $C$ such that, for any $f\in WH^p(\rn)$ and $\lz\in(0,\,\fz)$,
\begin{align*}
\lf|\lf\{x\in\rn: \ \mu_{\Omega,\,S}^\rho(f)(x)>\lz\r\}\r|\le C{\lz^{-p}}{\|f\|_{WH^p(\rn)}^p}.
\end{align*}
To this end, we choose integer $k_0$ satisfying $2^{k_0}\le\lz<2^{k_0+1}$.
By Lemma \ref{whp}, we may write
\begin{align*}
f=\sum_{k=-\fz}^{k_0}\sum_{i}b^k_i+\sum_{k=k_0+1}^{\fz}\sum_i b^k_i=:F_1+F_2,
\end{align*}
where $b^k_i$ satisfies (a), (b) and (c) of Lemma \ref{whp}.

We estimate $F_1$ first. For $F_1$, we claim that $\|F_1\|_{L^4(\rn)}\ls\lz^{1-{p}/{4}}\|f\|^{{p}/{4}}_{WH^p(\rn)}$.
In fact, Minkowski's inequality and the finite overlapped property of $\{B^k_i\}$ yield that
\begin{align*}
\|F_1\|_{L^4(\rn)}
&\le\sum_{k=-\fz}^{k_0}\sum_{i}\lf\|b^k_i\r\|_{L^4(\rn)}
\le\sum_{k=-\fz}^{k_0}\sum_{i}\lf\|b^k_i\r\|_{L^\fz(\rn)}\lf|B^k_i\r|^{1/4} \\
&\ls\sum_{k=-\fz}^{k_0}2^k\lf(\sum_{i}\lf|B^k_i\r|\r)^{1/4}\ls\sum_{k=-\fz}^{k_0}2^{k(1-{p}/{4})}\|f\|_{WH^p(\rn)}^{{p}/{4}}
\sim\lz^{(1-{p}/{4})}\|f\|_{WH^p(\rn)}^{{p}/{4}}.
\end{align*}
From the $L^4$ boundedness of $\mu_{\Omega,\,S}^\rho$ (see Theorem A) and the above claim, we deduce that
\begin{align*}
\lf|\lf\{x\in\rn: \ \mu_{\Omega,\,S}^\rho(F_1)(x)>\lz\r\}\r|
&\le \lz^{-4}\lf\|\mu_{\Omega,\,S}^\rho(F_1)\r\|^4_{L^4(\rn)} \\
&\ls \lz^{-4}\lf\|F_1\r\|^4_{L^4(\rn)}\ls\lz^{-p}\|f\|_{WH^p(\rn)}^p.
\end{align*}

Next let us deal with $F_2$. Set
\begin{align*}
A_{k_0}:=\bigcup_{k=k_0+1}^\fz\bigcup_i\tilde{B_i^k}\, ,
\end{align*}
where
$\tilde{B_i^k}:=B(x^k_i,\,64(3/2)^{{(k-k_0)p}/{n}}\,r^k_i)$.
To show that
\begin{align*}
\lf|\lf\{x\in\rn: \ \mu_{\Omega,\,S}^\rho(F_2)(x)>\lz\r\}\r|\ls{\lz^{-p}}{\|f\|_{WH^p(\rn)}^p},
\end{align*}
we cut $\{x\in\rn: \ \mu_{\Omega,\,S}^\rho(F_2)(x)>\lz\}$ into $A_{k_0}$ and
$\{x\in (A_{k_0})^\complement: \ \mu_{\Omega,\,S}^\rho(F_2)(x)>\lz\}$.

For $A_{k_0}$, a routine computation gives rise to
\begin{align*}
|A_{k_0}|
&\le \sum_{k=k_0+1}^{\fz}\sum_i\lf|\tilde{B_i^k}\r|\sim\sum_{k=k_0+1}^{\fz}\sum_i\lf(\frac{3}{2}\r)^{(k-k_0)p}\lf|{B_i^k}\r| \\
&\ls\sum_{k=k_0+1}^{\fz}\lf(\frac{3}{2}\r)^{(k-k_0)p}2^{-kp}\|f\|^p_{WH^p(\rn)}\sim\lz^{-p}\|f\|_{WH^p(\rn)}^p.
\end{align*}

It remains to estimate $(A_{k_0})^\complement$.
Applying the inequality $\|\cdot\|_{\ell^1}\le\|\cdot\|_{\ell^p}$ with $p\in(0,\,1]$, and Lemma \ref{yl.1}, we conclude that
\begin{align*}
&\lz^p\lf|\lf\{x\in{\lf(A_{k_0}\r)^\complement}: \ \mu_{\Omega,\,S}^\rho(F_2)(x)>\lz\r\}\r| \\
&\hs\le \int_{(A_{k_0})^\complement}\lf|\mu_{\Omega,\,S}^\rho(F_2)(x)\r|^p\,dx \\
&\hs\le \int_{(A_{k_0})^\complement}\sum_{k=k_0+1}^{\fz}\sum_i\lf|\mu_{\Omega,\,S}^\rho(b^k_i)(x)\r|^p\,dx \\
&\hs\le \sum_{k=k_0+1}^{\fz}\sum_i\int_{({\tilde{B_i^k}})^\complement}\lf|\mu_{\Omega,\,S}^\rho(b^k_i)(x)\r|^p\,dx \\
&\hs\ls \sum_{k=k_0+1}^{\fz}\sum_i\int_{({\tilde{B_i^k}})^\complement}
 \lf\|b^k_i\r\|^p_{L^\fz(\rn)}\frac{(r^k_i)^{(n+\bz)p}}{|x-x^k_i|^{(n+\bz)p}}\,dx \\
&\hs\ls \sum_{k=k_0+1}^{\fz}\sum_i2^{kp}\int_{|x-x^k_i|>(3/2)^{\frac{(k-k_0)p}{n}}r^k_i}
 \frac{(r^k_i)^{(n+\bz)p}}{|x-x^k_i|^{(n+\bz)p}}\,dx \\
&\hs\sim\sum_{k=k_0+1}^{\fz}\sum_i 2^{kp}\lf|B^k_i\r|\lf(\frac{2}{3}\r)^{\frac{p(np+\bz p-n)}{n}(k-k_0)} \\
&\hs\ls\|f\|^p_{WH^p(\rn)}\sum_{k=k_0+1}^{\fz}\lf(\frac{2}{3}\r)^{\frac{p(np+\bz p-n)}{n}(k-k_0)}\\
&\hs\sim\|f\|^p_{WH^p(\rn)},
\end{align*}
where the last ``$\sim$" is due to $p>n/(n+\bz)$. The proof is completed.
%\begin{align*}
%\mathrm{K_1}=\int_{\lf({\tilde{B_i^k}}\r)^\complement}\lf(\int_0^{|x-x^k_i|+2r^k_i}\lf|\int_{|x-y|\le t}
% \frac{\Omega(x,\,x-y)}{|x-y|^{n-\rho}}b^k_i(y)\,dy\r|^2\frac{dt}{t^{2\rho+1}}\r)^{{p}/{2}}\,dx
%\end{align*}
%and
%\begin{align*}
%\mathrm{K_2}=\int_{\lf({\tilde{B_i^k}}\r)^\complement}\lf(\int^\fz_{|x-x^k_i|+2r^k_i}\lf|\int_{|x-y|\le t}
% \frac{\Omega(x,\,x-y)}{|x-y|^{n-\rho}}b^k_i(y)\,dy\r|^2\frac{dt}{t^{2\rho+1}}\r)^{{p}/{2}}\,dx.
%\end{align*}
%
%The estimates of $\mathrm{K_1}$ and $\mathrm{K_2}$
%are quite similar to that given earlier for $\mathrm{J_1}$ and $\mathrm{J_2}$ in Theorem \ref{dl.1}, respectively,
%and hence no proof will be given here.
%We directly give the estimate of $\mathrm{K_1}+\mathrm{K_2}$ below,
%\begin{align*}
%\mathrm{K_1}+\mathrm{K_2}\ls2^{kp}\lf|B^k_i\r|\lf(\frac{2}{3}\r)^{\frac{p(pn+p\bz-n)}{n}(k-k_0)},
%\end{align*}
%which, together with $p>n/(n+\bz)$, implies that
%\begin{align*}
%\lz^p\lf|\lf\{x\in{\lf(A_{k_0}\r)^\complement}: \ \mu_\Omega^\rho(F_2)(x)>\lz\r\}\r|
%&\ls\sum_{k=k_0+1}^{\fz}\sum_i(\mathrm{K_1+K_{2}}) \\
%&\ls\sum_{k=k_0+1}^{\fz}\sum_i 2^{kp}\lf|B^k_i\r|\lf(\frac{2}{3}\r)^{\frac{p(np+\bz p-n)}{n}(k-k_0)} \\
%&\ls\|f\|^p_{WH^p(\rn)}\sum_{k=k_0+1}^{\fz}\lf(\frac{2}{3}\r)^{\frac{p(np+\bz p-n)}{n}(k-k_0)}\\
%&\sim\|f\|^p_{WH^p(\rn)}.
%\end{align*}
\end{proof}

\section{Boundedness of $\mu_{\Omega,\,\lz}^{\rho,\,\ast}$ on $H^p(\rn)$ and $WH^p(\rn)$}\label{s3}
The main results of this section are as follows.
\begin{theorem}\label{dl.4}
Let $0<\az\le1$, $n/2<\rho<n$, $2<\lz<\fz$, $0<\bz<\min\{1/2,\,\az,\,\rho-n/2,\,(\lz-2)n/3\}$ and $n/(n+\bz)<p<1$.
Suppose $\Omega(x,\,z)$ satisfies the $L^{2,\,\az}$-Dini condition or the Lipschitz condition of order $\az$.
Then $\mu_{\Omega,\,S}^\rho$ is bounded from $H^p(\rn)$ to $L^p(\rn)$.
\end{theorem}

%\begin{theorem}\label{dl.2}
%Let $0<\az\le1$, $n/2<\rho<n$, $0<\bz<\min\{1/2,\,\az,\,\rho-n/2\}$ and $n/(n+\bz)<p<1$.
%Suppose $\Omega$ satisfies the Lipschitz condition of order $\az$.
%Then $\mu_{\Omega,\,S}^\rho$ is bounded from $H^p(\rn)$ to $L^p(\rn)$.
%\end{theorem}

\begin{theorem}\label{dl.5}
Let $n/2<\rho<n$, $2<\lz<\fz$ and $\Omega(x,\,z)\in L^\fz(\rn) \times L^2(S^{n-1})$.
If
\begin{align*}
\int_0^1\frac{\omega_2(\delta)}{\delta}(1+|\log{\delta}|)^\sz\,d\delta<\fz \ \mathrm{for \ some} \ \sz>1,
\end{align*}
then $\mu^{\rho,\,\ast}_{\Omega,\,\lz}$ is bounded from $H^1(\rn)$ to $L^1(\rn)$.
\end{theorem}

\begin{theorem}\label{dl.6}
Let $0<\az\le1$, $n/2<\rho<n$, $2<\lz<\fz$, $0<\bz<\min\{1/2,\,\az,\,\rho-n/2,\,(\lz-2)n/3\}$ and $n/(n+\bz)<p\le1$.
Suppose $\Omega(x,\,z)$ satisfies the $L^{2,\,\az}$-Dini condition or the Lipschitz condition of order $\az$.
Then $\mu^{\rho,\,\ast}_{\Omega,\,\lz}$ is bounded from $WH^p(\rn)$ to $WL^p(\rn)$.
\end{theorem}

%
%\begin{corollary}\label{tl.7}
%Let $0<\rho<n$ and $q>2(n-1)/n$.
%Suppose $\Omega(x,\,z)\in L^\fz(\rn) \times L^q(S^{n-1})$. If \eqref{p2} or \eqref{p3} holds,
%then $\mu_\Omega^\rho$ is bounded from $L^1(\rn)$ to $WL^1(\rn)$.
%\end{corollary}

\begin{lemma}\label{yl.2}
Let $0<\az\le1$, $n/2<\rho<n$, $2<\lz<\fz$ and $0<\bz<\min\{1/2,\,\az,\,\rho-n/2,\,(\lz-2)n/3\}$.
Suppose $\Omega(x,\,z)$ satisfies the $L^{2,\,\az}$-Dini condition or the Lipschitz condition of order $\az$.
If $a(x)$ is a $(p,\,\fz,\,s)$-atom associated with some ball $B:=B(x_0,\,r)$,
then there exists a positive constant $C$ independent of $a(x)$ such that, for any $x\in (64B)^\complement$,
\begin{align*}
\mu^{\rho,\,\ast}_{\Omega,\,\lz}(a)(x)\le C\|a\|_{L^\fz(\rn)}\frac{r^{n+\beta}}{|x-x_0|^{n+\beta}}.
\end{align*}
\end{lemma}

\begin{proof}
We only consider the case $\Omega(x,\,z)$ satisfies the $L^{2,\,\az}$-Dini condition.
In another case, the proof is easier and we leave the details to the interested reader.
By Lemma \ref{yl.1}, we know that, for any $x\in (64B)^\complement$,
\begin{align*}
\mu^{\rho,\,\ast}_{\Omega,\,\lz}(a)(x)
&=\lf[\int\int_{{\mathbb{R}}^{n+1}_+}\lf(\frac{t}{t+|x-y|}\r)^{\lz n}\lf|\int_{|y-z|<t}\frac{\Omega(y,\,y-z)}{|y-z|^{n-\rho}}a(z)\,dz\r|^2\,\frac{dydt}{t^{n+2\rho+1}}\r]^{1/2} \\
&\le\lf[\int\int_{|y-x|<t}\lf(\frac{t}{t+|x-y|}\r)^{\lz n}\lf|\int_{|y-z|<t}\frac{\Omega(y,\,y-z)}{|y-z|^{n-\rho}}a(z)\,dz\r|^2\,\frac{dydt}{t^{n+2\rho+1}}\r]^{1/2} \\
&\hs+\lf[\int\int_{|y-x|\ge t}\lf(\frac{t}{t+|x-y|}\r)^{\lz n}\lf|\int_{|y-z|<t}\frac{\Omega(y,\,y-z)}{|y-z|^{n-\rho}}a(z)\,dz\r|^2\,\frac{dydt}{t^{n+2\rho+1}}\r]^{1/2} \\
&\le\mu^\rho_{\Omega,\,S}(a)(x) \\
&\hs+\lf[\int\int_{|y-x|\ge t}\lf(\frac{t}{t+|x-y|}\r)^{\lz n}\lf|\int_{|y-z|<t}\frac{\Omega(y,\,y-z)}{|y-z|^{n-\rho}}a(z)\,dz\r|^2\,\frac{dydt}{t^{n+2\rho+1}}\r]^{1/2} \\
&\le C\|a\|_{L^\fz(\rn)}\frac{r^{n+\beta}}{|x-x_0|^{n+\beta}} \\
&\hs+\lf[\int\int_{|y-x|\ge t}\lf(\frac{t}{t+|x-y|}\r)^{\lz n}\lf|\int_{|y-z|<t}\frac{\Omega(y,\,y-z)}{|y-z|^{n-\rho}}a(z)\,dz\r|^2\,\frac{dydt}{t^{n+2\rho+1}}\r]^{1/2} \\
&=:C\|a\|_{L^\fz(\rn)}\frac{r^{n+\beta}}{|x-x_0|^{n+\beta}}+{\rm{J}}.
\end{align*}
Thus, to show Lemma \ref{yl.2}, it suffices to prove that, for any $x\in (64B)^\complement$,
$${\rm{J}}\ls \|a\|_{L^\fz(\rn)}\frac{r^{n+\beta}}{|x-x_0|^{n+\beta}}.$$

For any $x\in (64B)^\complement$, write
\begin{align*}
{\rm{J}}
&\le \lf[\int\int_{\substack{|y-x|\ge t \\ y\in 16B}}
\lf(\frac{t}{t+|x-y|}\r)^{\lz n}\lf|\int_{|y-z|<t}\frac{\Omega(y,\,y-z)}{|y-z|^{n-\rho}}a(z)\,dz\r|^2\,\frac{dydt}{t^{n+2\rho+1}}\r]^{1/2} \\
&\hs+\lf[\int\int_{\substack{|y-x|\ge t \\ y\in (16B)^\complement \\ t\le|y-x_0|+8r}}\cdot\cdot\cdot\r]^{1/2}
+\lf[\int\int_{\substack{|y-x|\ge t \\ y\in (16B)^\complement \\ t>|y-x_0|+8r}}\cdot\cdot\cdot\r]^{1/2}=:{\rm{J_1+J_2+J_3}}.
\end{align*}

For ${\rm{J_1}}$, by $x\in (64B)^\complement$, $y\in 16B$ and $z\in B$, we know that
\begin{align*}
|x-x_0|/2<|x-y|<2|x-x_0| \ {\rm{and}} \ |y-z|<32r.
%|y-x|\sim|x-x_0|.
\end{align*}
From this, Minkowski's inequality for integrals, $0<\bz<\min\{\rho-n/2,\,(\lz-2)n/3\}$ and
$\Omega(x,\,z)\in L^\fz(\rn)\times L^2(S^{n-1})$,
it follows that, for any $x\in (64B)^\complement$,
\begin{align*}
{\rm{J_1}}
&\le\lf[\int\int_{\substack{|y-x|\ge t \\  |x-x_0|/2<|x-y|<2|x-x_0|}}\lf(\frac{t}{t+|x-y|}\r)^{\lz n}
\lf|\int_{\substack{|y-z|<t \\ |y-z|<32r}}\frac{\Omega(y,\,y-z)}{|y-z|^{n-\rho}}a(z)\,dz\r|^2\,\frac{dydt}{t^{n+2\rho+1}}\r]^{1/2} \\
&\le \int_B|a(z)|\lf[\int\int_{\substack{2|x-x_0|\ge t  \\ |x-x_0|/2<|x-y|<2|x-x_0| \\ |y-z|<32r,\,|y-z|<t}}
\lf(\frac{t}{t+|x-y|}\r)^{\lz n}\frac{|\Omega(y,\,y-z)|^2}{|y-z|^{2n-2\rho}}\frac{dydt}{t^{n+2\rho+1}}\r]^{1/2}dz \\
&\ls \int_B|a(z)|\lf[\int\int_{\substack{2|x-x_0|\ge t  \\ |y-z|<32r}}\lf(\frac{t}{|y-z|}\r)^{2\rho-n-2\bz}\r. \\
&\lf. \hspace{4.4 cm} \times \lf(\frac{t}{|x-x_0|}\r)^{2n+3\bz}\frac{|\Omega(y,\,y-z)|^2}{|y-z|^{2n-2\rho}}\frac{dydt}{t^{n+2\rho+1}}\r]^{1/2}dz \\
%&\ls \int_B|a(z)|\lf[\int\int_{\substack{y\in 16B \\ |y-x|\ge t \\ |y-x|>|x-x_0|/2 \\ |y-z|<32r,\,|y-z|<t}}
%\lf(\frac{t}{|x-x_0|}\r)^{2n+2\bz}\frac{1}{|y-z|^{n-\bz}t^{2n+\bz+1}}{dydt}\r]^{1/2}dz \\
&\ls \int_B|a(z)|\lf[\int_{|y-z|<32r}\frac{|\Omega(y,\,y-z)|^2}{|x-x_0|^{2n+3\bz}|y-z|^{n-2\bz}}\lf(\int_0^{2|x-x_0|}t^{\bz-1}dt\r)dy\r]^{1/2}dz \\
&\sim \int_B|a(z)|\lf(\int_{|y-z|<32r}\frac{|\Omega(y,\,y-z)|^2}{|x-x_0|^{2n+2\bz}|y-z|^{n-2\bz}}dy\r)^{1/2}dz \\
&\ls \frac{1}{|x-x_0|^{n+\bz}}\int_B|a(z)|\lf(\int^{32r}_0\frac{u^{n-1}}{u^{n-2\bz}}\,du\r)^{1/2}dz \\
&\ls \|a\|_{L^\fz(\rn)}\frac{r^{n+\bz}}{|x-x_0|^{n+\bz}},
\end{align*}
which is wished.

For ${\rm{J_2}}$, rewrite
\begin{align*}
{\rm{J_2}}
&\le\lf[\int\int_{\substack{|y-x|\ge t \\ y\in (16B)^\complement \\ t\le|y-x_0|+8r \\ |x-x_0|>3|y-x_0|}}
\lf(\frac{t}{t+|x-y|}\r)^{\lz n}\lf|\int_{|y-z|<t}\frac{\Omega(y,\,y-z)}{|y-z|^{n-\rho}}a(z)\,dz\r|^2\,\frac{dydt}{t^{n+2\rho+1}}\r]^{1/2} \\
&\hs+\lf[\int\int_{\substack{ y\in (16B)^\complement \\ t\le|y-x_0|+8r \\ |x-x_0|\le3|y-x_0|}}
\lf|\int_{|y-z|<t}\frac{\Omega(y,\,y-z)}{|y-z|^{n-\rho}}a(z)\,dz\r|^2\,\frac{dydt}{t^{n+2\rho+1}}\r]^{1/2}=:{\rm{J_{21}+J_{22}}}.
\end{align*}

The estimate of ${\rm{J_{22}}}$ is quite similar to that given earlier for ${\rm{I_{2}}}$ and so is omitted.

We are now turning to the estimate of ${\rm{J_{21}}}$.
By $x\in (64B)^\complement$, $y\in (16B)^\complement$, $z\in B$, $|x-x_0|>3|y-x_0|$ and the mean value theorem, we know that
\begin{align}\label{p1}
r<|y-z|\sim|y-x_0|;
\end{align}
\begin{align}\label{p4}
|x-y|\ge|x-x_0|-|y-x_0|> |x-x_0|/2;
\end{align}
\begin{align}\label{p2}
|y-x_0|-2r\le|y-x_0|-|x_0-z|\le|y-z|<t\le|y-x_0|+8r;
\end{align}
\begin{align}\label{p5}
\lf|\frac{1}{(|y-x_0|-2r)^{2\rho-n-2\bz}}-\frac{1}{(|y-x_0|+8r)^{2\rho-n-2\bz}}\r|\ls\frac{r}{|y-x_0|^{2\rho-n-2\bz+1}}.
\end{align}
From Minkowski's inequality for integrals, \eqref{p1}-\eqref{p5}, $\bz<\min\{1/2,\,(\lz-2)n/3\}<(\lz-2)n/2$
and $\Omega(x,\,z)\in L^\fz(\rn)\times L^2(S^{n-1})$, it follows that, for any $x\in (64B)^\complement$,
\begin{align*}
{\rm{J_{21}}}
&=\lf[\int\int_{\substack{|y-x|\ge t \\ y\in (16B)^\complement \\ t\le|y-x_0|+8r \\ |x-x_0|>3|y-x_0|}}
 \lf(\frac{t}{t+|x-y|}\r)^{\lz n}\lf|\int_{|y-z|<t}\frac{\Omega(y,\,y-z)}{|y-z|^{n-\rho}}a(z)\,dz\r|^2\,\frac{dydt}{t^{n+2\rho+1}}\r]^{1/2} \\
&\le \int_B|a(z)|\lf[\int\int_{\substack{|y-x|\ge t \\ y\in (16B)^\complement \\ t\le|y-x_0|+8r \\ |x-x_0|>3|y-x_0| \\ |y-z|<t }}
 \lf(\frac{t}{t+|x-y|}\r)^{2n+2\bz}\frac{|\Omega(y,\,y-z)|^2}{|y-z|^{2n-2\rho}}\frac{dydt}{t^{n+2\rho+1}}\r]^{1/2}dz \\
&\le \int_B|a(z)|\lf[\int\int_{\substack{ |y-z|>r \\ |x-y|>|x-x_0|/2 \\ |y-x_0|-2r\le t\le|y-x_0|+8r }}
 \lf(\frac{t}{|x-y|}\r)^{2n+2\bz}\frac{|\Omega(y,\,y-z)|^2}{|y-z|^{2n-2\rho}}\frac{dydt}{t^{n+2\rho+1}}\r]^{1/2}dz \\
&\ls \frac{\|a\|_{L^\fz(\rn)}}{|x-x_0|^{n+\bz}}\int_B\lf[\int_{|y-z|>r}
 \frac{|\Omega(y,\,y-z)|^2}{|y-z|^{2n-2\rho}}\lf(\int^{|y-x_0|+8r}_{|y-x_0|-2r}\frac{dt}{t^{2\rho-n-2\bz+1}}\r)dy\r]^{1/2}dz \\
&\ls \frac{\|a\|_{L^\fz(\rn)}}{|x-x_0|^{n+\bz}}\int_B\lf(\int_{|y-z|>r}
 \frac{|\Omega(y,\,y-z)|^2}{|y-z|^{2n-2\rho}}\frac{r}{|y-x_0|^{2\rho-n-2\bz+1}}\,dy\r)^{1/2}dz \\
&\sim {\|a\|_{L^\fz(\rn)}}\frac{r^{1/2}}{|x-x_0|^{n+\bz}}\int_B\lf(\int_{|y-z|>r}
 \frac{|\Omega(y,\,y-z)|^2}{|y-z|^{n-2\bz+1}}\,dy\r)^{1/2}dz \\
&\ls {\|a\|_{L^\fz(\rn)}}\frac{r^{1/2}}{|x-x_0|^{n+\bz}}\int_B\lf(\int_{r}^\fz
 \frac{du}{u^{1-2\bz+1}}\r)^{1/2}dz \\
&\sim {\|a\|_{L^\fz(\rn)}}\frac{r^{n+\bz}}{|x-x_0|^{n+\bz}},
\end{align*}
which is also wished.

%For ${\rm{J_{22}}}$, by repeating the estimate of ${\rm{I_{2}}}$ in the proof of Lemma \ref{m2}, we know that, for any $x\in (64B)^\complement$,
%\begin{align*}
%{\rm{J_{22}}}\ls \|b\|_{L^\fz}\frac{r^{n+\bz}}{|x-x_0|^{n+\bz}},
%\end{align*}
%which, together with ${\rm{J_{21}}}$, implies that
%\begin{align*}
%{\rm{J_{2}}}\ls \|b\|_{L^\fz}\frac{r^{n+\bz}}{|x-x_0|^{n+\bz}}.
%\end{align*}

For ${\rm{J_3}}$, noticing that $t>|y-x_0|+8r$, we see that, for any $y\in (16B)^\complement$,
\begin{align}\label{p6}
B\subset \{z\in\rn: \ |z-y|<t\};
\end{align}
\begin{align}\label{p7}
t+|x-y|\ge t+|x-x_0|-|y-x_0|\ge |x-x_0|+8r>|x-x_0|.
\end{align}
From \eqref{p6}, the vanishing moments of atom $a(z)$, Minkowski's inequality for integrals, \eqref{p7},
$\bz<\min\{\az,\,\rho-n/2,\,(\lz-2)n/3\}<(\lz-2)n/2$ and the argument same as in ${\rm{I_{31}}}$,
it follows that, for any $x\in (64B)^\complement$,
\begin{align*}
{\rm{J_{3}}}
&= \lf[\int\int_{\substack{|y-x|\ge t \\ y\in (16B)^\complement \\ t>|y-x_0|+8r}}
\lf(\frac{t}{t+|x-y|}\r)^{\lz n}\lf|\int_{|y-z|<t}\frac{\Omega(y,\,y-z)}{|y-z|^{n-\rho}}a(z)\,dz\r|^2\,\frac{dydt}{t^{n+2\rho+1}}\r]^{1/2} \\
&\le\int_B |a(z)|
\lf[\int\int_{\substack{|y-x|\ge t \\ y\in (16B)^\complement \\ t>|y-x_0|+8r \\ |y-z|<t \\ t+|x-y|>|x-x_0|}}
\lf(\frac{t}{t+|x-y|}\r)^{\lz n}
\lf|\frac{\Omega(y,\,y-z)}{|y-z|^{n-\rho}}-\frac{\Omega(y,\,y-x_0)}{|y-x_0|^{n-\rho}}\r|^2\frac{dydt}{t^{n+2\rho+1}}\r]^{1/2}dz \\
&\le\int_B |a(z)|
\lf[\int\int_{\substack{y\in (16B)^\complement \\ t>|y-x_0|  \\ t+|x-y|>|x-x_0|}}
\lf(\frac{t+|x-y|}{|x-x_0|}\r)^{2n+2\bz}\lf(\frac{t}{t+|x-y|}\r)^{\lz n} \r. \\
&\lf. \hspace{5.4 cm}\times\lf|\frac{\Omega(y,\,y-z)}{|y-z|^{n-\rho}}-\frac{\Omega(y,\,y-x_0)}{|y-x_0|^{n-\rho}}\r|^2\frac{dydt}{t^{n+2\rho+1}}\r]^{1/2}dz \\
&\le\frac{1}{|x-x_0|^{n+\bz}}\int_B |a(z)|
\lf[\int\int_{\substack{y\in (16B)^\complement \\ t>|y-x_0|}}
\frac{t^{\lz n}}{(t+|x-y|)^{\lz n-2n-2\bz}} \r. \\
&\lf. \hspace{5.4 cm}\times\lf|\frac{\Omega(y,\,y-z)}{|y-z|^{n-\rho}}-\frac{\Omega(y,\,y-x_0)}{|y-x_0|^{n-\rho}}\r|^2\frac{dydt}{t^{n+2\rho+1}}\r]^{1/2}dz \\
&\le\frac{\|a\|_{L^\fz(\rn)}}{|x-x_0|^{n+\bz}}\int_B
\lf[\int_{{y\in (16B)^\complement}}
\lf|\frac{\Omega(y,\,y-z)}{|y-z|^{n-\rho}}-\frac{\Omega(y,\,y-x_0)}{|y-x_0|^{n-\rho}}\r|^2\lf(\int^\fz_{|y-x_0|}\frac{dt}{t^{2\rho-n-2\bz+1}}\r)dy\r]^{1/2}dz \\
&\sim\frac{\|a\|_{L^\fz(\rn)}}{|x-x_0|^{n+\bz}}\int_B
\lf(\int_{y\in (16B)^\complement }
\lf|\frac{\Omega(y,\,y-z)}{|y-z|^{n-\rho}}-\frac{\Omega(y,\,y-x_0)}{|y-x_0|^{n-\rho}}\r|^2\frac{1}{|y-x_0|^{2\rho-n-2\bz}}dy\r)^{1/2}dz \\
&\ls {\|a\|_{L^\fz(\rn)}}\frac{r^{n+\bz}}{|x-x_0|^{n+\bz}}.
%&\ls\frac{r^\az}{|x-x_0|^{n+\bz}}\int_B |a(z)|
%\lf[\int\int_{\substack{|y-x|\ge t \\ y\in (16B)^\complement \\ t>|y-x_0|+8r \\ |y-z|<t \\ t+|x-y|>|x-x_0|}}
%\frac{t^{\lz n}}{(t+|x-y|)^{\lz n-2n-2\bz}} \r.\\
%&\lf. \hspace{7.3 cm} \times\frac{1}{|y-x_0|^{2n-2\rho+2\az}}\frac{dydt}{t^{n+2\rho+1}}\r]^{1/2}dz \\
%&\ls\frac{r^\az}{|x-x_0|^{n+\bz}}\int_B |a(z)|
%\lf(\int\int_{\substack{ y\in B^\complement \\ t>|y-x_0|}}
%\frac{1}{|y-x_0|^{2n-2\rho+2\az}}\frac{dydt}{t^{-n+2\rho-2\bz+1}}\r)^{1/2}dz \\
%&\ls\|b\|_{L^\fz}\frac{r^{n+\az}}{|x-x_0|^{n+\bz}}
%\lf[\int_{B^\complement}\frac{1}{|y-x_0|^{2n-2\rho+2\az}}\lf(\int_{|y-x_0|}^\fz
%\frac{dt}{t^{-n+2\rho-2\bz+1}}\r)dy\r]^{1/2}  \\
%&\sim\|b\|_{L^\fz}\frac{r^{n+\az}}{|x-x_0|^{n+\bz}}
%\lf(\int_{B^\complement}\frac{1}{|y-x_0|^{n+2\az-2\bz}}dy\r)^{1/2} \\
%&\sim \|b\|_{L^\fz}\frac{r^{n+\az}}{|x-x_0|^{n+\bz}}\lf(\int_{S^{n-1}}\int^\fz_r\frac{1}{s^{n+2\az-2\bz}}s^{n-1}\,dsd\sigma(y')\r)^{1/2} \\
%&\sim \|b\|_{L^\fz}\frac{r^{n+\az}}{|x-x_0|^{n+\bz}}\,r^{\bz-\az}\sim\|b\|_{L^\fz}\frac{r^{n+\bz}}{|x-x_0|^{n+\bz}}.
\end{align*}

Combining the estimates of ${\rm{J_1}}$, ${\rm{J_{21}}}$, ${\rm{J_{22}}}$ and ${\rm{J_3}}$, we obtain the desired inequality.
This finishes the proof of Lemma \ref{yl.2}.
\end{proof}

\begin{proof}[Proof of Theorems \ref{dl.4} and \ref{dl.6}]
Once we prove the Lemma \ref{yl.2}, the proofs of Theorems \ref{dl.4} and \ref{dl.6}
are identity to that of Theorems \ref{dl.1} and \ref{dl.3}, respectively, the details being omitted.
\end{proof}

\begin{proof}[Proof of Theorem \ref{dl.5}]
Proceeding as in the proof of \cite[Theorem 1.1]{dlx07lp},
it is quite believable that \cite[Theorem 1.1]{dlx07lp} may also be true for the variable kernel case,
but to limit the length of this paper, we leave the details to the interested reader.
\end{proof}

%\begin{proof}[Proof of Corollary \ref{tl.7}]
%By an argument similar to that used in \cite[Remark 1.8]{dls04},
%we can easily carry out the proof of this corollary, the details being omitted.
%\end{proof}

%%%%%%%%%%%%%%%%%%%%%%%%%%%%%%%%%%%%%%%%%%%%%%%%%%%%%%%%%%%%%%%%%%%%%%%

%%%%%%%%%%%%%%%%%%%%%%%settion 4 %%%%%%%%%%%%%%%%%%%%%%%%%%%%%%%%%%%%%

%%%%%%%%%%%%%%%%%%%%%%%%%%%%%%%%%%%%%%%%%%%%%%%%%%%%%%%%%%%%%%%%%%%%%%

\section{Final remark}\label{s4}
We conclude this paper by pointing out some remarks.

First of all, the weak-type space plays very important role in harmonic analysis since it
can sharpen the endpoint weak type estimate for variant important operators.
Therefore, with the help of Lemmas \ref{yl.1} and \ref{yl.2}, we can easily carry out the proof of following two theorems.
But to limit the length of this paper, we leave the details to the interested reader.
\begin{theorem}\label{dl.11}
Let $0<\az\le1$, $n/2<\rho<n$ and $0<\bz<\min\{1/2,\,\az,\,\rho-n/2\}$.
Suppose $\Omega(x,\,z)$ satisfies the $L^{2,\,\az}$-Dini condition or the Lipschitz condition of order $\az$.
Then $\mu_{\Omega,\,S}^\rho$ is bounded from $H^{\frac{n}{n+\bz}}(\rn)$ to $WL^{\frac{n}{n+\bz}}(\rn)$.
\end{theorem}
\begin{theorem}\label{dl.12}
Let $0<\az\le1$, $n/2<\rho<n$, $2<\lz<\fz$ and $0<\bz<\min\{1/2,\,\az,\,\rho-n/2,\,(\lz-2)n/3\}$.
Suppose $\Omega(x,\,z)$ satisfies the $L^{2,\,\az}$-Dini condition or the Lipschitz condition of order $\az$.
Then $\mu^{\rho,\,\ast}_{\Omega,\,\lz}$ is bounded from $H^{\frac{n}{n+\bz}}(\rn)$ to $WL^{\frac{n}{n+\bz}}(\rn)$.
\end{theorem}

Secondly, by using the interpolation theorem of sublinear operator (see \cite[p. 63]{l95})
between Theorem \ref{dl.2} (resp. Theorem \ref{dl.5}) and Theorem A, we get immediately the following
$L^p$ boundedness of $\mu_{\Omega,\,S}^\rho$ (resp. $\mu^{\rho,\,\ast}_{\Omega,\,\lz}$) for $1<p<4$.
\begin{corollary}\label{tl.1}
Let $1<p<4$, $n/2<\rho<n$ and $\Omega(x,\,z)\in L^\fz(\rn) \times L^2(S^{n-1})$.
If
\begin{align*}
\int_0^1\frac{\omega_2(\delta)}{\delta}(1+|\log{\delta}|)^\sz\,d\delta<\fz \ \mathrm{for \ some} \ \sz>1,
\end{align*}
then $\mu_{\Omega,\,S}^\rho$ is bounded on $L^p(\rn)$.
\end{corollary}

\begin{corollary}\label{tl.2}
Let $1<p<4$, $n/2<\rho<n$, $2<\lz<\fz$ and $\Omega(x,\,z)\in L^\fz(\rn) \times L^2(S^{n-1})$.
If
\begin{align*}
\int_0^1\frac{\omega_2(\delta)}{\delta}(1+|\log{\delta}|)^\sz\,d\delta<\fz \ \mathrm{for \ some} \ \sz>1,
\end{align*}
then $\mu^{\rho,\,\ast}_{\Omega,\,\lz}$ is bounded is bounded on $L^p(\rn)$.
\end{corollary}

And lastly,, by the $WH^1$-$WL^1$ boundedness of $\mu_{\Omega,\,S}^\rho$ and $\mu^{\rho,\,\ast}_{\Omega,\,\lz}$ (see Theorems \ref{dl.3} and \ref{dl.6})
and the argument same as in \cite[Remark 3]{dlx07a} (see also \cite[Remark 1.7]{dlx07lp}),
we have the following two corollaries.

\begin{corollary}\label{tl.3}
Let $0<\az\le1$, $n/2<\rho<n$.
Suppose $\Omega(x,\,z)$ satisfies the $L^{2,\,\az}$-Dini condition or the Lipschitz condition of order $\az$.
Then $\mu_{\Omega,\,S}^\rho$ is bounded from $WH^1(\rn)$ to $WL^1(\rn)$.
\end{corollary}

\begin{corollary}\label{tl.4}
Let $0<\az\le1$, $n/2<\rho<n$, $2<\lz<\fz$.
Suppose $\Omega(x,\,z)$ satisfies the $L^{2,\,\az}$-Dini condition or the Lipschitz condition of order $\az$.
Then $\mu^{\rho,\,\ast}_{\Omega,\,\lz}$ is bounded from $WH^1(\rn)$ to $WL^1(\rn)$.
\end{corollary}

%{\small\noindent{\bf Acknowledgements}\quad
%The author would like to thank the referees
%for their very carefully reading
%and so many useful remarks which did make this article more readable.}

%\bigskip
%
%\noindent  Liu Xiong, Li Baode, Qiu Xiaoli and Li Bo (Corresponding author)
%
%\medskip
%
%\noindent
%College of Mathematics and System Sciences \\
%Xinjiang University\\
%Urumqi 830046\\
%P. R. China
%
%\smallskip
%
%\noindent{E-mail }:\\
%\texttt{1394758246@qq.com} (Liu Xiong)  \\
%\texttt{1246530557@qq.com} (Li Baode)  \\
%\texttt{2237424863@qq.com} (Qiu Xiaoli) \\
%\texttt{bli.math@outlook.com} (Li Bo)

%\bigskip \medskip

\noindent Li Bo
%
%\medskip
%
%\noindent
%School of Mathematics \\
%Tianjin University \\
%Tianjin 300354\\
%P. R. China
%
%\smallskip

\noindent{E-mail }: \texttt{bli.math@outlook.com}

\bigskip


\begin{thebibliography}{30}
%\vspace{-0.3cm}
%\bibitem{agd14}
%Akbulut, Ali; Guliyev, Vagif Sabir; Dziri, Moncef:
%Weighted norm inequalities for the $g$-Littlewood-Paley operators associated with Laplace-Bessel differential operators.
%Math. Inequal. Appl. 17 (2014), no. 1, 317-333.

%\vspace{-0.3cm}
%\bibitem{bckyy13}
%Bui, The Anh; Cao, Jun; Ky, Luong Dang; Yang, Dachun; Yang, Sibei:
%Musielak-Orlicz-Hardy spaces associated with operators satisfying reinforced off-diagonal estimates.
%Anal. Geom. Metr. Spaces 1 (2013), 69-129.

%\vspace{-0.3cm}
%\bibitem{fl16}
%Fan, Xingya; Li, Baode,
%Anisotropic tent spaces of Musielak-Orlicz type and their applications,
%Adv. Math. (China) 45 (2016), no. 2, 233-251.

%\vspace{-0.3cm}
%\bibitem{ccyy16}
%Cao, Jun; Chang, Der-Chen; Yang, Dachun; Yang, Sibei:
%Riesz transform characterizations of Musielak-Orlicz-Hardy spaces.
%Trans. Amer. Math. Soc. 368 (2016), no. 10, 6979-7018.

%\vspace{-0.3cm}
%\bibitem{ccyy13}
%Cao, Jun; Chang, Der-Chen; Yang, Dachun; Yang, Sibei:
%Weighted local Orlicz-Hardy spaces on domains and their applications in inhomogeneous Dirichlet and Neumann problems.
%Trans. Amer. Math. Soc. 365 (2013), no. 9, 4729-4809.
%
%\vspace{-0.3cm}
%\bibitem{cw77}
%Coifman, Ronald R.; Weiss, Guido:
%Extensions of Hardy spaces and their use in analysis.
%Bull. Amer. Math. Soc. 83 (1977), no. 4, 569-645.

%\vspace{-0.3cm}
%\bibitem{cxy15}
%Chen, Xi; Xue, Qingying; Yabuta, K\^{o}z\^{o}:
%On multilinear Littlewood-Paley operators.
%Nonlinear Anal. 115 (2015), 25-40.

\vspace{-0.3cm}
\bibitem{cz55}
Calder\'{o}n, Alberto Pedro; Zygmund, Antoni Szczepan:
On a problem of Mihlin.
Trans. Amer. Math. Soc. 78, (1955), 209-224.

\vspace{-0.3cm}
\bibitem{cz56}
Calder\'{o}n, Alberto Pedro; Zygmund, Antoni Szczepan:
On singular integrals.
Amer. J. Math. 78 (1956), 289-309.

%\vspace{-0.3cm}
%\bibitem{cz78}
%Calder\'{o}n, Alberto Pedro; Zygmund, Antoni Szczepan:
%On singular integrals with variable kernels.
%Applicable Anal. 7 (1977/78), no. 3, 221-238.
%\vspace{-0.3cm}
%\bibitem{c74}
%Coifman, R.R.: A real variable characterization of H p. Studia Math. 51, 269-274 (1974)

%\vspace{-0.3cm}
%\bibitem{d05}
%Diening, Lars: Maximal function on Musielak-Orlicz spaces and generalized Lebesgue spaces.
%Bull. Sci. Math. 129 (2005), no. 8, 657-700.
%
%\vspace{-0.3cm}
%\bibitem{dhr09}
%Diening, Lars; H\"{a}st\"{o}; Peter A.; Roudenko, Svetlana: Function spaces of variable smoothness and integrability.
%J. Funct. Anal. 256 (2009), no. 6, 1731-1768.

%\vspace{-0.3cm}
%\bibitem{dlx07a}
%Ding, Yong; Lu, Shan Zhen; Xue, Qing Ying:
%Parametrized area integrals on Hardy spaces and weak Hardy spaces.
%Acta Math. Sin. (Engl. Ser.) 23 (2007), no. 9, 1537-1552.
%
%\vspace{-0.3cm}
%\bibitem{dlx07lp}
%Ding, Yong; Lu, Shanzhen; Xue, Qingying:
%Parametrized Littlewood-Paley operators on Hardy and weak Hardy spaces.
%Math. Nachr. 280 (2007), no. 4, 351-363.

%\vspace{-0.3cm}
%\bibitem{cch15}
%Chen, Jiecheng; Chen, Yanping; Hu, Guoen:
%Compactness for the commutators of singular integral operators with rough variable kernels.
%J. Math. Anal. Appl. 431 (2015), no. 1, 597-621.


\vspace{-0.3cm}
\bibitem{cd11}
Chen, Yanping; Ding, Yong:
$L^p$ boundedness for Littlewood-Paley operators with rough variable kernels.
J. Math. Anal. Appl. 377 (2011), no. 2, 889-904.

%\vspace{-0.3cm}
%\bibitem{cd08}
%Chen, Yanping; Ding, Yong:
%Boundedness of commutators of Marcinkiewicz integral with rough variable kernel.
%Integral Equations Operator Theory 61 (2008), no. 4, 477-492.

%\vspace{-0.3cm}
%\bibitem{dls04}
%Ding, Yong; Lin, Chin-Cheng; Shao, Shuanglin:
%On the Marcinkiewicz integral with variable kernels.
%Indiana Univ. Math. J. 53 (2004), no. 3, 805-821.

\vspace{-0.3cm}
\bibitem{dll07}
Ding, Yong; Lin, Chin-Cheng; Lin, Ying-Chieh:
Erratum: "On Marcinkiewicz integral with variable kernels'' [Indiana Univ. Math. J. 53 (2004), no. 3, 805-821; MR2086701] by Ding, C.-C. Lin and S. Shao.
Indiana Univ. Math. J. 56 (2007), no. 2, 991-994.

\vspace{-0.3cm}
\bibitem{dlx07a}
Ding, Yong; Lu, Shanzhen; Xue, Qingying:
Parametrized area integrals on Hardy spaces and weak Hardy spaces.
Acta Math. Sin. (Engl. Ser.) 23 (2007), no. 9, 1537-1552.

\vspace{-0.3cm}
\bibitem{dlx07lp}
Ding, Yong; Lu, Shanzhen; Xue, Qingying:
Parametrized Littlewood-Paley operators on Hardy and weak Hardy spaces.
Math. Nachr. 280 (2007), no. 4, 351-363.

%\vspace{-0.3cm}
%\bibitem{ty02}
%Tang, Lin; Yang, Dachun:
%Boundedness of singular integrals of variable rough Calder\'{o}n-Zygmund kernels along surfaces.
%Integral Equations Operator Theory 43 (2002), no. 4, 488-502.

\vspace{-0.3cm}
\bibitem{l95}
Lu, Shanzhen:
Four Lectures on Real $H^p$ Spaces
(World Scientific Publishing Co., Inc., River Edge, NJ, 1995).
%\vspace{-0.3cm}
%\bibitem{l88}
%Liu, Heping:
%The weak $H^p$ spaces on homogeneous groups.
%Harmonic analysis (Tianjin, 1988), 113-118, Lecture Notes in Math., 1494, Springer, Berlin, 1991.

%\vspace{-0.3cm}
%\bibitem{fs71}
%Fefferman, Charles Louis; Stein, Elias M.:
%Some maximal inequalities.
%Amer. J. Math. 93 (1971), 107-115.
%
%\vspace{-0.3cm}
%\bibitem{fs72}
%Fefferman, Charles Louis; Stein, Elias M.:
%$H^p$ spaces of several variables.
%Acta Math. 129 (1972), no. 3-4, 137-193.
%
%
%
%\vspace{-0.3cm}
%\bibitem{fhly17}
%Fan, Xingya; He, Jianxun; Li, Baode; Yang, Dachun:
%Real-variable characterizations of anisotropic product Musielak-Orlicz Hardy spaces.
%Sci. China Math. 60 (2017), no. 11, 2093-2154.
%
%\vspace{-0.3cm}
%\bibitem{hyy14}
%Hou, Shaoxiong; Yang, Dachun; Yang, Sibei:
%Musielak-Orlicz BMO-type spaces associated with generalized approximations to the identity.
%Acta Math. Sin. (Engl. Ser.) 30 (2014), no. 11, 1917-1962.

%\vspace{-0.3cm}
%\bibitem{fjw91}
%Frazier, Michael; Jawerth, Bj\"{o}rn; Weiss, Guido:
%Littlewood-Paley theory and the study of function spaces.
%CBMS Regional Conference Series in Mathematics, 79. Published for the Conference Board of the Mathematical Sciences,
%Washington, DC; by the American Mathematical Society, Providence, RI, 1991. viii+132 pp.


%\vspace{-0.3cm}
%\bibitem{g09c}
%Grafakos, Loukas:
%Classical Fourier Analysis.
%Second edition, Graduate Texts in Mathematics Vol. 249 (Springer, New York, 2009).

%\vspace{-0.3cm}
%\bibitem{g09m}
%Grafakos, Loukas,
%Modern Fourier Analysis.
%Second edition, Graduate Texts in Mathematics Vol. 250 (Springer, New York, 2009).

%\vspace{-0.3cm}
%\bibitem{h14}
%Hardy, Godfrey Harold:
%The mean value of the modulus of an analytic function.
%Proc. Lond. Math. Soc. 14, (1914) 269-277.



%\vspace{-0.3cm}
%\bibitem{hxmy15}
%He, Sha; Xue, Qingying; Mei, Ting; Yabuta, K\^{o}z\^{o}:
%Existence and boundedness of multilinear Littlewood-Paley operators on Campanato spaces.
%J. Math. Anal. Appl. 432 (2015), no. 1, 86-102.
%\vspace{-0.3cm}
%\bibitem{j80}
%Janson, Svante:
%Generalizations of Lipschitz spaces and an application to Hardy spaces and bounded mean oscillation,
%Duke Math. J. 47 (1980), no. 4, 959-982.



%\vspace{-0.3cm}
%\bibitem{jn87}
%Johnson, Raymond L.; Neugebauer, Christoph Johannes:
%Homeomorphisms preserving $A_p$.
%Rev. Mat. Iberoamericana 3 (1987), no. 2, 249-273.

%\vspace{-0.3cm}
%\bibitem{jy10}
%Jiang, Renjin; Yang, Dachun:
%New Orlicz-Hardy spaces associated with divergence form elliptic operators.
%J. Funct. Anal. 258 (2010), no. 4, 1167-1224.
%
%\vspace{-0.3cm}
%\bibitem{jy11jfaa}
%Jiang, Renjin; Yang, Dachun:
%Predual spaces of Banach completions of Orlicz-Hardy spaces associated with operators.
%J. Fourier Anal. Appl. 17 (2011), no. 1, 1-35.
%
%\vspace{-0.3cm}
%\bibitem{jy11ccm}
%Jiang, Renjin; Yang, Dachun:
%Orlicz-Hardy spaces associated with operators satisfying Davies-Gaffney estimates.
%Commun. Contemp. Math. 13 (2011), no. 2, 331-373.
%
%\vspace{-0.3cm}
%\bibitem{k14}
%Ky, Luong Dang:
%New Hardy spaces of Musielak-Orlicz type and boundedness of sublinear operators.
%Integral Equations Operator Theory 78 (2014), no. 1, 115-150.
%
%\vspace{-0.3cm}
%\bibitem{k80}
%Kalton, Nigel John:
%Linear operators on $L_p$ for $0<p<1$.
%Trans. Amer. Math. Soc. 259 (1980), no. 2, 319-355.

%\vspace{-0.3cm}
%\bibitem{l17}
%Liu, Feng:
%On the Triebel-Lizorkin space boundedness of Marcinkiewicz integrals along compound surfaces.
%Math. Inequal. Appl. 20 (2017), no. 2, 515-535.
%
%\vspace{-0.3cm}
%\bibitem{lw14}
%Liu, Feng; Wu, Huoxiong:
%On the $L^2$ boundedness for the multiple Littlewood-Paley functions with rough kernels.
%J. Math. Anal. Appl. 410 (2014), no. 1, 403-410.

%\vspace{-0.3cm}
%\bibitem{lwz15}
%Liu, Feng; Wu, Huoxiong; Zhang, Daiqing:
%$L^p$ bounds for parametric Marcinkiewicz integrals with mixed homogeneity.
%Math. Inequal. Appl. 18 (2015), no. 2, 453-469.

%\vspace{-0.3cm}
%\bibitem{l77}
%Latter, R.H.: A decomposition of H p(Rn) in terms of atoms. Studia Math. 62, 92¨C101 (1977)


%\vspace{-0.3cm}
%\bibitem{lfy15}
%Li, Baode; Fan, Xingya; Yang, Dachun:
%Littlewood-Paley characterizations of anisotropic Hardy spaces of Musielak-Orlicz type.
%Taiwanese J. Math. 19 (2015), no. 1, 279-314.
%
%\vspace{-0.3cm}
%\bibitem{lffy16}
%Li, Baode; Fan, Xingya; Fu, Zunwei; Yang, Dachun:
%Molecular characterization of anisotropic Musielak-Orlicz Hardy spaces and their applications.
%Acta Math. Sin. (Engl. Ser.) 32 (2016), no. 11, 1391-1414.
%
%\vspace{-0.3cm}
%\bibitem{lll17}
%Li, Bo; Liao, Minfeng; Li, Baode:
%Boundedness of Marcinkiewicz integrals with rough kernels on Musielak-Orlicz Hardy spaces.
%J. Inequal. Appl. 2017, 2017: 228.

%\vspace{-0.3cm}
%\bibitem{llly08}
%Lee, Ming-Yi; Lin, Chin-Cheng; Lin, Ying-Chieh; Yan, Dunyan:
%Boundedness of singular integral operators with variable kernels.
%J. Math. Anal. Appl. 348 (2008), no. 2, 787-796.

%\vspace{-0.3cm}
%\bibitem{lsl16}
%Li, Jinxia; Sun, Ruirui; Li, Baode:
%Anisotropic interpolation theorems of Musielak-Orlicz type.
%J. Inequal. Appl. 2016, 2016: 243.

%\vspace{-0.3cm}
%\bibitem{lhy12}
%Liang, Yiyu; Huang, Jizheng; Yang, Dachun:
%New real-variable characterizations of Musielak-Orlicz Hardy spaces.
%J. Math. Anal. Appl. 395 (2012), no. 1, 413-428.
%
%
%
%\vspace{-0.3cm}
%\bibitem{ly13}
%Liang, Yiyu; Yang, Dachun:
%Musielak-Orlicz Campanato spaces and applications.
%J. Math. Anal. Appl. 406 (2013), no. 1, 307-322.
%
%\vspace{-0.3cm}
%\bibitem{lyj16}
%Liang, Yiyu; Yang, Dachun; Jiang, Renjin:
%Weak Musielak-Orlicz Hardy spaces and applications.
%Math. Nachr. 289 (2016), no. 5-6, 634-677.
%
%\vspace{-0.3cm}
%\bibitem{lnyz14}
%Liang, Yiyu; Nakai, Eiichi; Yang, Dachun; Zhang, Junqiang:
%Boundedness of intrinsic Littlewood-Paley functions on Musielak-Orlicz Morrey and Campanato spaces.
%Banach J. Math. Anal. 8 (2014), no. 1, 221-268.

%\vspace{-0.3cm}
%\bibitem{l95}
%Lu, Shanzhen:
%Four Lectures on Real $H^p$ Spaces.
%(World Scientific Publishing Co., Inc., River Edge, NJ, 1995).%Singapore




%\vspace{-0.3cm}
%\bibitem{s84}
%Stefan, Rolewicz:
%Metric Linear Spaces, Second edition.
%(PWN-Polish Scientific Publishers, Warsaw, 1984).
%
%
%\vspace{-0.3cm}
%\bibitem{st89}
%Str\"omberg, Jan-Olov; Torchinsky, Alberto:
%Weighted Hardy spaces.
%Lecture Notes in Mathematics Vol. 1381 (Springer-Verlag, Berlin, 1989).
%
%
%\vspace{-0.3cm}
%\bibitem{stw81}
%Stein, Elias M.; Taibleson, Mitchell H.; Weiss, Guido:
%Weak type estimates for maximal operators on certain $H^p$ classes.
%Proceedings of the Seminar on Harmonic Analysis (Pisa, 1980). Rend. Circ. Mat. Palermo 2 (1981), suppl. 1, 81-97.

%\vspace{-0.3cm}
%\bibitem{sxy14}
%Shi, Shaoguang; Xue, Qingying; Yabuta, K\^{o}z\^{o}:
%On the boundedness of multilinear Littlewood-Paley $g_\lambda^\ast$ function.
%J. Math. Pures Appl. (9) 101 (2014), no. 3, 394-413.


%\vspace{-0.3cm}
%\bibitem{sy99}
%Sakamoto, Minako; Yabuta, K\^{o}z\^{o}:
%Boundedness of Marcinkiewicz functions.
%Studia Math. 135 (1999), no. 2, 103-142.

%
%\vspace{-0.3cm}
%\bibitem{xd07}
%Xue, Qingying; Ding, Yong:
%Weighted $L^p$ boundedness for parametrized Littlewood-Paley operators.
%Taiwanese J. Math. 11 (2007), no. 4, 1143-1165.
%
%\vspace{-0.3cm}
%\bibitem{xyy15}
%Xue, Qingying; Yabuta, K\^{o}z\^{o}; Yan, Jingquan:
%On the boundedness of fractional type Marcinkiewicz integral operators.
%Math. Inequal. Appl. 18 (2015), no. 2, 519-527.

%\vspace{-0.3cm}
%\bibitem{y17}
%Yang, Anming:
%Atomic decompositions of vector-valued weak Musielak-Orlicz martingale spaces and applications.
%J. Math. Anal. Appl. 449 (2017), no. 1, 670-681.

%\vspace{-0.3cm}
%\bibitem{y15}
%Yang, Sibei:
%Several estimates of Musielak-Orlicz-Hardy-Sobolev type for Schr\"{o}dinger type operators.
%Ann. Funct. Anal. 6 (2015), no. 3, 118-144.
%
%\vspace{-0.3cm}
%\bibitem{ylk17}
%Yang, Dachun; Liang, Yiyu; Ky, Luong Dang:
%Real-Variable Theory of Musielak-Orlicz Hardy spaces.
%Lecture Notes in Mathematics Vol. 2182 (Springer, Cham, 2017).
%
%
%
%
%\vspace{-0.3cm}
%\bibitem{yyz14}
%Yang, Dachun; Yuan, Wen; Zhuo, Ciqiang:
%Musielak-Orlicz Besov-type and Triebel-Lizorkin-type spaces.
%Rev. Mat. Complut. 27 (2014), no. 1, 93-157.
%
%\vspace{-0.3cm}
%\bibitem{yy12}
%Yang, Dachun; Yang, Sibei:
%Local Hardy spaces of Musielak-Orlicz type and their applications.
%Sci. China Math. 55 (2012), no. 8, 1677-1720.
%
%\vspace{-0.3cm}
%\bibitem{yy14}
%Yang, Dachun; Yang, Sibei:
%Musielak-Orlicz-Hardy spaces associated with operators and their applications.
%J. Geom. Anal. 24 (2014), no. 1, 495-570.
%
%\vspace{-0.3cm}
%\bibitem{yy16}
%Yang, Dachun; Yang, Sibei:
%Maximal function characterizations of Musielak-Orlicz-Hardy spaces associated to non-negative self-adjoint operators satisfying Gaussian estimates.
%Commun. Pure Appl. Anal. 15 (2016), no. 6, 2135¨C2160.
%
%\vspace{-0.3cm}
%\bibitem{zql17}
%Zhang, Hui; Qi, Chunyan; Li, Baode:
%Anisotropic weak Hardy spaces of Musielak-Orlicz type and their applications.
%Front. Math. China 12 (2017), no. 4, 993-1022.


\end{thebibliography}
\end{document}